\documentclass{article}
\usepackage[utf8]{inputenc}

\usepackage{amsmath,mathtools,amsthm,a4wide}
\usepackage{amssymb, authblk}
\usepackage[mathscr]{eucal}
\usepackage{bm}
\usepackage{bbm}
\usepackage[shortlabels]{enumitem}
\usepackage{hyperref}
\usepackage[usenames,dvipsnames]{xcolor}
\usepackage[a4paper, total={6in, 8in}]{geometry}

\usepackage{tikz,pgfplots}
\usepackage{graphicx, caption,subcaption}
\usepackage{todonotes}
\usepackage{booktabs}
\usepackage{multirow}

\usepackage{siunitx}

\newcommand{\prob}{\mathbb{P}}
\newcommand{\Prob}[1]{\prob\left(#1\right)}

\newcommand{\expec}{\mathbb{E}}
\newcommand{\Exp}[1]{\expec\left[#1\right]}

\newcommand{\plim}{\ensuremath{\stackrel{\prob}{\longrightarrow}}}

\newcommand{\ind}[1]{\mathbbm{1}_{\left\{#1\right\}}}

\newcommand{\dd}{{\rm d}}

\allowdisplaybreaks

\newtheorem{theorem}{Theorem}[section]

\newtheorem{lemma}[theorem]{Lemma}

\newtheorem{corollary}[theorem]{Corollary}

\makeatletter
\patchcmd{\@sect}{#8}{\boldmath #8}{}{}
\let\ori@chapter\@chapter
\def\@chapter[#1]#2{\ori@chapter[\boldmath#1]{\boldmath#2}}
\makeatother

\title{Maximal Cliques in Scale-Free Random Graphs}
\author{Thomas Bl\"asius and Maximillian Katzmann\\\vspace{-0.3cm} {\small
Karlsruhe Institute of Technology, Karlsruhe, Germany}
\\\vspace{0.1cm} and \\\vspace{0.1cm} Clara Stegehuis \\ {\small
Faculty of Electrical Engineering, Mathematics and Computer Science,
University of Twente}}
\date{}

\begin{document}
\maketitle
\vspace{-1.1cm}
\begin{abstract}
  \noindent
  We investigate the number of maximal cliques, i.e., cliques that are
  not contained in any larger clique, in three network models:
  Erdős--Rényi random graphs, inhomogeneous random graphs (also called
  Chung--Lu graphs), and geometric inhomogeneous random graphs.  For
  sparse and not-too-dense Erdős--Rényi graphs, we give linear and
  polynomial upper bounds on the number of maximal cliques.  For the
  dense regime, we give super-polynomial and even exponential lower
  bounds.  Although (geometric) inhomogeneous random graphs are
  sparse, we give super-polynomial lower bounds for these models.
  This comes from the fact that these graphs have a power-law degree
  distribution, which leads to a dense subgraph in which we find many
  maximal cliques.  These lower bounds seem to contradict previous
  empirical evidence that (geometric) inhomogeneous random graphs have
  only few maximal cliques.  We resolve this contradiction by
  providing experiments indicating that, even for large networks, the
  linear lower-order terms dominate, before the super-polynomial
  asymptotic behavior kicks in only for networks of extreme size.
\end{abstract}

\section{Introduction}
	
While networks appear in many different applications, many real-world networks were found to share some important characteristics. First of all, often their degree distribution is heavy-tailed, which is sometimes denoted as the network being scale-free. Secondly, they often have a high clustering coefficient, implying that it is likely that two neighbors of a vertex are connected themselves as well. For this reason, random graph models that can achieve both scale-freeness and a high clustering coefficient have been at the center of attention over the last years. 
	
One example of such a model is the popular hyperbolic random graph (HRG)~\cite{krioukov2010}, which has for example been used to model the network of world wide trade~\cite{garcia-perez2016} or the Internet on the Autonomous Systems level~\cite{boguna2010,kleinberg2007}. This random graph model embeds the vertices in an underlying hyperbolic space and connects them with probabilities depending on their distances, where nearby vertices are more likely to connect. The triangle inequality then ensures the presence of many triangles, while the hyperbolic space ensures the presence of a scale-free degree distribution. Recently, the geometric inhomogeneous random graph (GIRG) was proposed as a generalization of HRG.  It combines power-law distributed weights with Euclidean space, making the model simpler to analyze~\cite{bringmann2015}. 
	
While the hyperbolic random graph and the GIRG have been designed to exhibit high clustering and a scale-free degree distribution, the question remains whether other properties of this model match real-world data. For this reason, many properties of the GIRG or hyperbolic random graph have been analyzed mathematically, such as the maximum clique size~\cite{Blasius2017}, number of $k$-cliques~\cite{michielan2021}, spectral gap~\cite{kiwi2018} and separator size~\cite{Hyper_Rando_Graph_Separ_ESA2016,lengler2017}. 
	
In this paper, we focus on another network property: the number of maximal cliques, i.e., cliques that are not part of any larger clique. 
Cliques in general are an important indicator for structural properties of a network. Indeed, the number of large cliques is a measure of the tendency of a network to cluster into groups. Small cliques of size 3 (triangles) on the other hand, can form an indication of the transitivity of a network or its clustering coefficient.
	
To study these structural clique-based properties, however, all cliques of a given size need to be listed, which can be a computationally expensive process. To list all network cliques, it suffices to list only all maximal cliques, as all smaller cliques can be generated from at least one maximal clique. For this reason, enumerating all maximal cliques of a graph is at the heart of our understanding of cliques in general. 

 
For enumerating all maximal cliques, an output-polynomial algorithm~\cite{tsukiyama1977} exists, which can enumerate all maximal cliques efficiently if the graph contains only few of them. This creates a link between enumeration and counting: if the maximal clique count is low, then it is possible to efficiently enumerate them. There also exist highly efficient implementations to enumerate all maximal cliques~\cite{Listi_Maxim_Cliqu_Spars_ISAAC2010,Listi_Maxim_Cliqu_Large_jour2013,Listi_Maxim_Cliqu_Large_SEA2011}. However, for a given graph, it is usually not known a priori how many maximal cliques it has. If this number is large, enumerating all maximal cliques can still take exponential time. However, in practice, enumerating the number of maximal cliques often takes a short amount of time for many real-world instances as well as in realistic network models~\cite{Exter_Valid_Avera_Analy_ESA2022}. 
In this paper, we therefore focus on the number of maximal cliques in the GIRG random graph, that is, the maximal clique count. As the GIRG possesses the two main characteristics that are essential to many real-world networks, scale-freeness and an underlying geometry, we believe that investigating the number of maximal cliques in the GIRG can provide insights into in why enumerating the number of maximal cliques can often be done efficiently for many real-world networks. 
	
To investigate the influence of the different properties of
scale-freeness and clustering, we investigate the number of maximal
cliques in three steps. First, we investigate a model without
heavy-tailed degrees and with a small clustering coefficient, the
Erd\H{o}s--R\'enyi model $G(n, p)$; see Section~\ref{sec:ER}. We then
investigate the GIRG model (Section~\ref{sec:GIRG}), which has both
clustering and scale-free degrees. Finally, in Section~\ref{sec:IRG},
we investigate the Inhomogeneous Random Graph (IRG), a model that is
scale-free but has a small clustering coefficient.  We complement our
theoretical bounds with experiments in Section~\ref{sec:experiments}.
In all models, we will be interested in the large $n$ limit. That is, we investigate how the number of maximal cliques scales in the number of nodes $n$ when $n$ grows large.
Our main findings can be summarized as follows; also see
Table~\ref{tab:result_summary_detailed} for an overview of our
results.
\begin{itemize}
\item There is a strong dependence on the density of the network.  For
  the Erdős--Rényi model ($G(n, p)$) we obtain a linear upper bound
  for sparse graphs ($O(n)$ edges) and a polynomial upper bound for
  non-dense graphs ($O(n^{2 - \varepsilon})$ edges for any
  $\varepsilon > 0$).  For dense graphs on the other hand
  ($\Omega(n^2)$ edges), we obtain a super-polynomial lower bound.  If
  the density is high enough, our lower bound is even exponential.
\item This insight carries over to the IRG and GIRG models.  Though
  they are overall sparse, they contain sufficiently large dense
  subgraphs that allow us to obtain super-polynomial lower bounds.
\item In the IRG model with power-law exponent $\tau \in (2, 3)$ the
  small maximal cliques localize: asymptotically maximal cliques of
  constant size $k>2$ are formed by $k - 2$ hubs of high degree
  proportional to $n^{1 / (\tau - 1)}$ and two vertices of lower
  degree proportional to $n^{(\tau-2)/(\tau-1)}$.
\item We complement our theoretical lower bounds with experiments
  showing that the super-{po\-ly\-no\-mi\-al} growth becomes only relevant for
  very large networks. 
\end{itemize}

\paragraph{Discussion and Related Work.}

Although cliques themselves have been studied extensively in the
literature, there is, to the best of our knowledge, only little
previous work on the number of \emph{maximal} cliques in network
models.  In fact, the only theoretical analysis we are aware of is the
recent preprint by Yamaji~\cite{yamaji2023}, giving bounds for
hyperbolic random graphs (HRG) and random geometric graph (RGG), which
are also shown in Table~\ref{tab:result_summary_detailed}.
Interestingly, this includes the upper bound of
$\exp(O(n^{\frac{3 - \tau}{6} + \varepsilon}))$ for the HRG model.  In
contrast to that, we give the asymptotically larger lower bound
$\exp(\Omega(n^{\frac{3 - \tau}{4} - \varepsilon}))$ for the
corresponding GIRG variant.  Thus, there is an asymptotic difference
between the HRG and the GIRG model.

This is surprising as the GIRG model is typically perceived as a
generalization of the HRG model.  More precisely, there is a mapping
between the two models such that for every HRG with average degree
$d_{\mathrm{HRG}}$ there exist GIRGs with average degree
$d_{\mathrm{GIRG}}$ and $D_{\mathrm{GIRG}}$ with
$d_{\mathrm{GIRG}} \le d_{\mathrm{HRG}} \le D_{\mathrm{GIRG}}$ that
are sub- and supergraphs of the HRG, respectively.  Moreover,
$d_{\mathrm{GIRG}}$ and $D_{\mathrm{GIRG}}$ are only a constant factor
apart and experiments indicate that
$d_{\mathrm{HRG}} = d_{\mathrm{GIRG}}\cdot(1 + o(1))$, i.e., every HRG has a
corresponding GIRG that is missing only a sublinear number of
edges~\cite{Effic_gener_geome_inhom_jour2022}.  In the case of maximal
cliques, however, this minor difference between the models leads to an
asymptotic difference.

Besides this theoretical analysis, it has been observed empirically
that the number of maximal cliques in most real-world networks as well
as in the GIRG and the IRG model is smaller than the number of edges
of the graph~\cite{Exter_Valid_Avera_Analy_ESA2022}.  This indicates
linear scaling in the graph size with low constant factors and small
lower-order terms, which seems to be a stark contradiction to the
super-polynomial lower bounds we prove here.  We resolve this
contradiction with our experiments in Section~\ref{sec:experiments},
where we observe that the graph size has to be quite large before the
asymptotic behavior kicks in, i.e., we observe the super-polynomial
scaling as predicted by our theorems but on such a low level that it
is overshadowed by the linear lower-order terms.

\paragraph{Notation and setting.} In the rest of this paper, we will be interested in results in the large $n$ limit, where $n$ denotes the number of nodes in the random graph. We therefore use classical asymptotic notation, in terms of the graph size $n$. For any two non-negative functions $f(n),g(n)$ we will write $f(n) \in o(g(n))$ if $\lim_{n \to \infty} f(n)/g(n) = 0$; $f(n) \in O(g(n))$ if $\limsup_{n \to \infty} f(n)/g(n) < \infty$; 
$f(n) \in \Omega(g(n))$ if $\liminf_{n \to \infty} f(n)/g(n) > 0$; 
$f(n) \in \Theta(g(n))$ if $f(n) \in O(g(n))$ and $f(n) \in \Omega(g(n))$. Moreover, we will say that a sequence of events $\{\mathcal{E}_n\}_{n \geq 1}$ happens with high probability (w.h.p.) if $\lim_{n \to \infty} \Prob{\mathcal{E}_n}=1$.





\begin{table}
  \newcommand{\rot}[2]{\multirow{#1}{*}{\rotatebox[origin=c]{90}{#2}}}
  \renewcommand{\arraystretch}{1.3}
  \centering
  \begin{tabular}{clcr}
    \toprule
    \multicolumn{2}{c}{Model} & Maximal cliques                                             & Reference                                                                                     \\
    \midrule
    \rot{4}{$G(n, p)$}        & $p = 1 - \Theta(\frac{1}{n})$                               & $2^{\Omega(n)}$                                       & Theorem~\ref{thm:erdense}             \\
                              & $p \in \Theta(1)$                                             & $n^{\Omega(\log n)}$                                  & Theorem~\ref{thm:er-lower-bound}      \\
                              & $p \in O(\frac{1}{n^a})$                                      & $n^{O(1)}$                                            & Theorem~\ref{thm:sparseer}            \\
                              & $p \in O(\frac{1}{n})$                                        & $O(n)$                                                & Theorem~\ref{thm:sparseer}            \\
    \midrule
    \multicolumn{2}{l}{IRG}   & $\exp(\Omega(n^{\frac{3 - \tau}{4} - \varepsilon} \log n))$ & Theorem~\ref{thm:max_cliques_irg}                                                             \\
    \midrule
    \rot{4}{GIRG}             & $d$-dim~torus, $T = 0$                                      & $\exp(\Omega(n^{\frac{3 - \tau}{4} - \varepsilon}))$  & Corollary~\ref{cor:maxcliquesgirg}    \\
                              & $d$-dim~torus, $T > 0$
                                                                                            & $\exp(\Omega(  {n^{\frac{(3 - \tau)}{5}} (\varepsilon \log n)^{-(1/2)}} ))$ & Corollary~\ref{cor:maxcliquesgirg-non-threshold} \\
                              & $2$-dim~square, $T = 0$                                     & $\exp(\Omega(n^{\frac{3 - \tau}{10} - \varepsilon}))$ & Theorem~\ref{thm:girg_non_torus}      \\
                              & $2$-dim~square, $T > 0$                                     & $\exp(\Omega(  n^{\frac{3 - \tau}{10} - \varepsilon}))$ & Theorem~\ref{thmnonzero2dim}          \\
    \midrule
    \rot{2}{RGG}              & 2-dim, dense                                                & $\exp(\Omega(n^{\frac{1}{3}}))$                       & \cite{yamaji2023}                     \\
                              & 2-dim, dense                                                & $\exp(O(n^{\frac{1}{3} + \varepsilon}))$              & \cite{yamaji2023}                     \\
    \midrule    
    \rot{2}{HRG}              &                                                             & $\exp(\Omega(n^{\frac{3 - \tau}{6}}))$                & \cite{yamaji2023}                     \\
                              &                                                             & $\exp(O(n^{\frac{3 - \tau}{6} + \varepsilon}))$       & \cite{yamaji2023}                     \\
    \bottomrule
  \end{tabular}
  \caption{Summary of our and other results on the number of maximal cliques in different random graph models and their scaling in the number of vertices. }
  \label{tab:result_summary_detailed}
\end{table}

\section{Erd\H{o}s--R\'enyi Random Graph }
\label{sec:ER}

An Erdős--Rényi random graph~\cite{Gilbert1959, Erdoes2022} $G(n, p)$ has $n$ vertices and each pair
of vertices is connected independently with probability $p$.  We give
bounds on the number of maximal cliques in a $G(n, p)$ depending on
$p$.  Roughly speaking, we give super-polynomial lower bounds for the
dense regime and polynomial upper bounds for a sparser regime.
Specifically, we first give a general lower bound that is
super-polynomial if $p$ is non-vanishing for growing $n$, i.e., if
$p \in \Omega(1)$.  Note that $p \in \Omega(1)$ yields a dense graph
with a quadratic number of edges in expectation.  For super-dense
graph with $p = 1 - c/n$ for a constant $c$, we strengthen this lower
bound to exponential.  In contrast to this, we give a polynomial upper
if $p \in O(n^{-a})$ for any constant $a > 0$.  For sparse graphs with
$p \in O(n^{-1})$, yielding graphs with $\Theta(n)$ edges in
expectation, our upper bound on the number of maximal cliques is
linear.  We start with the general lower bound.

\begin{theorem}
  \label{thm:er-lower-bound}
  Let $N$ be the number of maximal cliques in a $G(n, p)$.  Then, for $n$ sufficiently large,
  \begin{equation}
    \Exp{N} \ge n^{\frac{\log(n) / 2 - \log\log n + \log\log(1/p)}{\log(1/p)}}\cdot \frac{1 - o(1)}{e}.
  \end{equation}
\end{theorem}
\begin{proof}
  Let $N_k$ be the number of maximal cliques of size $k$.  To estimate
  $\Exp{N_k}$, note that the probability that a fixed subset
  $C \subseteq V$of $|C| = k$ vertices forms a clique is
  $p^{k(k-1)/2}$.  Moreover, it is maximal if none of the other
  $n - k$ vertices is connected to all $k$ vertices of $C$, which
  happens with probability $(1 - p^k)^{n - k}$.  As the two events are
  independent and there are ${n \choose k}$ vertex sets of size $k$,
  we obtain
  \begin{equation}
    \Exp{N_k} = {n\choose k}p^{k(k-1)/2}(1-p^k)^{n-k}.
  \end{equation}
  Using that ${n\choose k} \ge (n / k)^k$ and increasing the exponents of the probabilities, we obtain
  \begin{equation*}
    \Exp{N_k} \ge \left( \frac{n}{k} \right)^kp^{k^2/2}(1-p^k)^{n}.
  \end{equation*}

  We now set $k=\log(n)/\log(1/p) = -\log(n)/\log(p)$, which yields $p^k = n^{-1}$.  Thus, in the above bound, the term $n^k p^{k^2/2}$ simplifies to $n^k n^{-k/2} = n^{k/2}$.  Moreover, the term $(1 - p^k)^{n}$ simplifies to $(1 - 1/n)^n$, which converges to $1/e$ for $n \to \infty$.  Thus, we obtain
  \begin{align*}
    \Exp{N_k} &\ge n^{k/2} \frac{1}{e k^k} (1 - o(1))\\
              &= n^{\frac{\log(n) / 2}{\log(1/p)}} \cdot \left( \frac{\log(n)}{\log(1/p)} \right)^{-\frac{\log(n)}{\log(1/p)}} \cdot \frac{1 - o(1)}{e}.\\
    \intertext{Changing the base of the second factor yields}
              &= n^{\frac{\log(n) / 2}{\log(1/p)}} \cdot e^{-\frac{\log(n)}{\log(1/p)} \cdot \log\left( \frac{\log(n)}{\log(1/p)} \right)} \cdot \frac{1 - o(1)}{e}\\
              &= n^{\frac{\log(n) / 2}{\log(1/p)}} \cdot n^{-\frac{\log\log n - \log\log(1/p)}{\log(1/p)}} \cdot \frac{1 - o(1)}{e}\\
              & = n^{\frac{\log(n) / 2 - \log\log n + \log\log(1/p)}{\log(1/p)}}\cdot \frac{1 - o(1)}{e}.
  \end{align*}
  As there are clearly at least as many maximal cliques as maximal cliques of size $k$, claimed bound for $\Exp{N}$ follows.
\end{proof}

This means that in a dense Erd\H{o}s--R\'enyi random graph (constant $p$), the expected number of maximal cliques is super-polynomial in $n$.  In the following, we show that, when the graph gets even denser, the number of maximal cliques even grows exponentially.  For this, we prove the existence of an induced subgraph that has many maximal cliques.  Specifically, we aim to find a large \emph{co-matching}, i.e., the complement graph of a matching (or equivalently, a co-matching).

\begin{lemma}
  \label{lem:co-matching-nr-cliques}
  Let $G$ be a co-matching on $2k$ vertices.  Then $G$ has $2^k$ maximal cliques.
\end{lemma}
\begin{proof}
  The complement $\overline G$ of $G$ is a matching with $k$ edges.  The maximal independent sets of $\overline G$ are the vertex sets that contain for each edge exactly one of its vertices.  Thus, $\overline G$ has $2^k$ maximal independent sets, which implies that $G$ has $2^k$ maximal cliques.
\end{proof}

With this, we can show an exponential lower bound for super-dense Erdős--R\'enyi graphs.

\begin{theorem}\label{thm:erdense}
  For every $c>0$, there exists a $\zeta>0$ and $n'>0$ such that $G(n,1-c/n)$ contains at least $2^{\zeta n}$ cliques with high probability for all $n\geq n'$. 
\end{theorem}
\begin{proof}
  A co-matching in $G(n,1-c/n)$ corresponds to an induced matching in $G(n,c/n)$. Now fix $M>1$. Then, by~\cite[Theorem 5.12]{hofstad2009}, with high probability the Erd\H{o}s--R\'enyi random graph contains a linear number of vertices of degree at most $M$ and at least $1$. Denote the reduced graph with only vertices of degree at most $M$ by $G_{\leq M}$, which has a linear number of edges. Now we construct an induced matching of linear size in $G_{\leq M}$ as follows. Start with any edge $\{u,v\}$ in $G_{\leq M}$, and add it to the matching. Then, remove $u$, $v$ and all neighbors of $u$ and $v$ from $G_{\leq M}$. This removes at most $2M^2$ edges from $G_{\leq M}$, as all degrees are bounded by $M$. Then, pick another edge and continue this process until $G_{\leq M}$ contains no more edges. As this process removes only a constant number of edges after picking a new edge, at least a linear number of edges will be added before the process finishes. Thus, there is an induced matching of at least $\zeta n$ with high probability, which yields the claim due to Lemma~\ref{lem:co-matching-nr-cliques}.    
\end{proof}

Next we consider less dense Erdős--Rényi graphs with $p \in O(n^{-a})$
for a constant $a \in (0, 1]$ and prove a polynomial upper bound on
the number of maximal cliques.  The degree of the polynomial depends
on $a$.  For sparse graphs with $p \in O(n^{-1})$, our bound is
linear.
		
\begin{theorem}
  \label{thm:sparseer}
  Let $p = (c / n)^a$ for constants $c > 0$ and $a \in (0, 1]$ and let
  $N$ be the number of maximal cliques in a $G(n, p)$.  Then
  $\Exp{N} \in O(n^x)$ with
  \begin{equation*}
    x = \left\lceil \frac{1}{a}\right\rceil - a \cdot {\left\lceil
        \frac{1}{a}\right\rceil \choose 2}.
  \end{equation*}
\end{theorem}
\begin{proof}
  As in Theorem~\ref{thm:er-lower-bound}, let $N_k$ be the number of
  maximal cliques of size $k$.  Note that the number of maximal
  cliques is upper bounded by the number of (potentially non-maximal)
  cliques.  Thus, we obtain
  \begin{equation*}
    \Exp{N_k} \le {n\choose k}p^{\frac{k(k-1)}{2}}.
  \end{equation*}
  Using that ${n \choose k} \le (en / k)^k$, inserting
  $p = (c / n)^a$, and rearranging yields
  \begin{align}
    \Exp{N_k}
    &\le \left( \frac{en}{k} \right)^k
      \left( \frac{c}{n} \right)^{a\frac{k(k-1)}{2}}\notag\\
    &= \left( \frac{ce}{k} \right)^k
      \left( \frac{c}{n} \right)^{a\frac{k(k-1)}{2} - k}.
      \label{eq:er-upper-bound-expectation}
  \end{align}

  We first argue that we can focus on the case where $k$ is constant
  as the above term vanishes sufficiently quickly for growing $k$.
  For this, note that $a {k(k-1)} / {2} - k \ge k$ if $k \ge 4/a + 1$.
  Thus, as $c / n < 1$ for sufficiently large $n$, the second factor
  of Equation~\eqref{eq:er-upper-bound-expectation} is upper bounded
  by $(c / n)^k$.  For $k \ge 4/a + 1$, it then follows that
  $\Exp{N_k} \le ( {c^2e}/ (kn) )^k$.  For sufficiently large $n$, the
  fraction is smaller than $1$ and thus the sum over all $N_k$ for
  larger values of $k$ is upper bounded by a constant due to the
  convergence of the geometric series.

  Focusing on $k \in \Theta(1)$ and ignoring constant factors, we
  obtain
  \begin{equation*}
    \Exp{N} \in O\left( \max_{k \in \mathbb N^+} \left\{n^{x(k)}\right\} \right)
    \text{ with }
    x(k) = {k - a\frac{k(k-1)}{2}}.
  \end{equation*}

  To evaluate the maximum, note that $x(k)$ describes a parabola with
  its maximum at $k_0 = 1/a + 1/2$.  However, $k_0$ may not be
  integral.  To determine the integer $k$ that maximizes $x(k)$, note
  that for $a \in [\frac{1}{i}, \frac{1}{i - 1}]$ with $i \in \mathbb
  N^+$, we get $k_0 \in [i - \frac{1}{2}, i + \frac{1}{2}]$.  Thus,
  $i$ is the closest integer to $k_0$.  As the parabola is symmetric
  at its maximum $k_0$, the exponent $x(k)$ is maximized for the
  integer $k = i = \lceil \frac{1}{a} \rceil$.  Substituting $k(k-1)/2
  = {k \choose 2}$ yields the claim.
\end{proof}

 
	\section{Geometric Inhomogeneous Random Graphs (GIRG)}\label{sec:GIRG}
	While the Erd\H{o}s--R\'enyi random graph is homogeneous, and does not contain geometry, we now investigate the number of maximal cliques in a model that contains both these properties, the Geometric Inhomogeneous Random Graph (GIRG)~\cite{bringmann2015}. We will use similar notation as in~\cite{bringmann2015}, except for the parameters $\alpha$ and $\beta$, which we will replace by $1/T$ and $\tau$ respectively, to be more consistent with the literature on other similar models~\cite{krioukov2010}. In this model, each vertex $v$ has a weight, $w_v$ and a position $x_v$.
	The weights are independent copies of a power-law random variable $W$ with exponent $\tau$, i.e.,
	\begin{equation}\label{eq:pl}
		1-F(w):=\mathbb{P}(W > w) = 
		w^{1-\tau}, 
	\end{equation} 
	for all $w\geq 1$. We impose the condition $\tau \in (2,3)$, to ensure that the weights have finite mean but unbounded variance. The parameter $\mu$ denotes the mean of this distribution, and can be computed as $\mu = (\tau-2)^{-1}$.
	The vertex positions $x_1,...,x_n$ are independent copies of a uniform random variable on the $d$-dimensional torus $\mathbb{T}^d = \mathbb{R}^d/\mathbb{Z}^d$.
	
	An edge between any two vertices $u,v \in V$ of the GIRG appears independently with a probability $p_{uv}$ determined by the weights and the positions of the vertices
	\begin{equation} \label{eq:edgeprob}
		p_{uv}  = \min \left\{\left(\frac{w_u w_v}{n\mu\|x_u-x_v\|^{d}}\right)^{1/T}, 1 \right\},
	\end{equation}
        where $\|\cdot\|$ denotes the maximum norm on the torus, $\mu$ is a parameter controlling the average degree, and $0 < T < 1$ is the \emph{temperature} and controls the influence of the geometry. We say that $T=0$ is the \emph{threshold} case of the GIRG. That is, when $T=0$,
	\begin{equation} \label{eq:edgeprobthreshold}
		p_{uv}  =\begin{cases}
			1 & \frac{w_u w_v}{n\mu\|x_u-x_v\|^{d}}\geq 1 \\
			0 &\text{else}.
		\end{cases}
\end{equation}
In general, we will be interested in results for the GIRG model when the number of nodes, $n$, tends to infinity. We will then often refer to the family of GIRGs generated for varying $n$ by $G^{(n)}$, where we assume that all other parameters ($\mu,\ \tau,\ \gamma,\  d$) remain fixed.

In the following, we first give a lower bound for the threshold case
(Section~\ref{sec:girg-threshold-case}).  The proof makes use of the
toroidal structure of the ground space.  To prove that this is not
essential to obtain a super-polynomial number of maximal cliques, we
additionally give a lower bound for a variant of the model where the
ground space is a 2-dimensional unit square with Euclidean norm
(Section~\ref{sec:girg-with-2d-square}).  Finally, in
Section~\ref{sec:girg-non-threshold-case}, we show how to extend these
results to the general case with non-zero temperatures.

	
\subsection{Threshold Case}
\label{sec:girg-threshold-case}

Here we show that a $d$-dimensional threshold GIRG $G = (V, E)$ has,
with high probability, a super-polynomial number of maximal cliques.
To achieve this, we proceed as follows to show that $G$ has a large
co-matching as induced subgraph (also see
Lemma~\ref{lem:co-matching-nr-cliques}).  We consider the vertex set
$S \subseteq V$ containing all vertices whose weight lies between a
lower bound $w_\ell$ and an upper bound $w_u$.  As a co-matching is
quite dense, it makes sense to think of these as rather large weights.
We then define disjoint regions $B_1, \dots, B_{2k}$.  For
$i \in [k]$, we call $B_i$ and $B_{i + k}$ a pair of \emph{opposite}
regions.  These regions will satisfy the following three properties.
First, every $B_i$ contains a vertex from $S$ with high probability.
Secondly, pairs of vertices from $S$ in opposite regions are not
connected.  And thirdly, vertices from $S$ that do not lie in opposite
regions are connected.  Note that these properties imply the existence
of a co-matching on $2k$ vertices, as choosing an arbitrary vertex of
$S$ for each region $B_i$ makes it so that each chosen vertex has
exactly one partner from the opposite region to which it is not
connected, while it is connected to the vertices from all other
regions.
	
	\begin{figure}
		\centering
		\hfill
		\begin{subfigure}{0.48\linewidth}
			\centering
			\includegraphics[page=1]{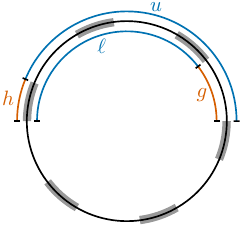}
			\caption{$d = 1$}
			\label{fig:girg-torus-1d}
		\end{subfigure}
		\hfill
		\begin{subfigure}{0.48\linewidth}
			\centering
			\includegraphics[page=2]{girg-torus}
			\caption{$d = 2$}
			\label{fig:girg-torus-2d}
		\end{subfigure}
		\hfill
		\caption{Illustration of the gray shaded boxes $B_i$ on the 1 and 2-dimensional torus.}
		\label{fig:girg-torus}
	\end{figure}
	
	In the following we first give a parameterized definition of the
	regions $B_i$ and then show how to choose the parameters for the above
	strategy to work; also see Figure~\ref{fig:girg-torus}.  Each $B_i$ is
	an axis-aligned box, i.e., the cross product of intervals.  Let
	$g(n), h(n) > 0$ such that $1 / (g(n) + h(n))$ is an even number.  Think of $h(n)$ of
	as the \emph{height} of each box and of $g(n)$ as the \emph{gap} between
	the boxes, yielding $2k = 1 / (g(n) + h(n))$ boxes.  Now we define
	$B_i = [(i - 1) \cdot (g(n) + h(n)), (i - 1) \cdot (g(n) + h(n)) + h(n)] \times [0,
	\frac{1}{2} - g(n)]^{d - 1}$ for $i \in [2k]$.  We call the resulting
	regions $B_1, \dots, B_{2k}$ the \emph{evenly spaced boxes of height
		$h(n)$ and gap $g(n)$}, see Figure~\ref{fig:girg-torus} for an illustration for $d=1$ and $d=2$.  As before, $B_i$ and $B_{i + k}$ for $i \in [k]$
	are \emph{opposite} boxes.
	
	With this, note that the distance between any pair of points in
	opposite boxes is at least $u = \frac{1}{2} - h(n)$ (recall that we
	assume the infinity norm).  Moreover, the distance between any pair of
	points in non-opposite regions is at most $\ell = \frac{1}{2} - g(n)$.
	This yields the following lemma.
	
	\begin{lemma}
		\label{lem:evenly-spaced-boxes-co-matching}
		Let $B_1, \dots, B_{2k}$ be evenly spaced boxes of height $h(n)$ and
		gap $g(n)$ in $\mathbb T^d$.  Let
		$w_\ell = (\frac{1}{2} - g(n))^{d / 2} \sqrt{\mu n}$ and
		$w_u = (\frac{1}{2} - h(n))^{d / 2} \sqrt{\mu n}$.  If a GIRG on $n$ vertices with $T=0$, $\tau\in(2,3)$ and $\mu$ places one
		vertex of weight in $[w_\ell, w_u)$ in each box $B_i$, then these
		vertices form a co-matching.
	\end{lemma}
	\begin{proof}
		As observed above, the vertices in opposite boxes have distance at
		least $u = \frac{1}{2} - h(n)$.  Moreover, the vertices considered here
		have weight less than $w_u = u^{d / 2} \sqrt{\mu n}$.  As
		$w_u^2 / (\mu n u^d) = 1$, these vertices are not connected because the weight interval $[w_\ell, w_u)$ is open at $w_u$ (see
		Equation~\eqref{eq:edgeprobthreshold}).  Similarly, vertices in
		non-opposite boxes have distance at most $\ell = \frac{1}{2} - g(n)$
		and weight at least $w_\ell = \ell^{d / 2} \sqrt{\mu n}$.  As
		$w_\ell^2 / (\mu n \ell^d) = 1$, such vertices are connected.
		Hence, we get a co-matching.
	\end{proof}
	
	It now remains to choose $g(n)$ and $h(n)$ appropriately.  First observe
	that, for the weight range in
	Lemma~\ref{lem:evenly-spaced-boxes-co-matching} to be non-empty, we
	need $w_\ell < w_u$ and thus $g > h$.  Beyond that, we want to achieve
	the following three goals.  First, the weight range needs to be
	sufficiently large such that we actually have a sufficient number of
	vertices in this range.  For this, we want to choose $g(n)$ substantially
	larger than $h(n)$. Secondly, we want to make each box $B_i$
	sufficiently large for it to contain a vertex with high probability.
	For this, we mainly want $h(n)$ to be large. Thirdly, we want the number
	of boxes $2k = 1 / (g(n) + h(n))$ to be large to obtain a large co-matching.
	For this, we want $g(n)$ and $h(n)$ to be small.
	
	Note that the restrictions of choosing $h(n)$ large, $g(n)$ larger than
	$h(n)$, and $g(n) + h(n)$ small are obviously conflicting.  In the following,
	we show how to balance these goals out to obtain a co-matching of
	polynomial size.  We start by estimating the number of vertices in the
	given weight range in the following lemma, which is slightly more
	general then we need.
	
	\begin{lemma}
		\label{lem:nr-vertices-in-weight-range}
		Let the vertex weights independently be sampled as in~\eqref{eq:pl}, with $\tau\in(2,3)$. Let $a, b > 0$ be constants and let $g(n), h(n)$ be functions of $n$ such
		that $g(n), h(n) \in o(1)$.  Let $S$ be the set of vertices with weight in
		$[(a - g(n))^b \sqrt{\mu n}, (a - h(n))^b \sqrt{\mu n})$.  Then
		\begin{equation}
			\Exp{|S|} =
			n^{\frac{3 - \tau}{2}} \cdot
			\mu^{\frac{1 - \tau}{2}} ba^{b(1 - \tau) - 1} \cdot
			(g(n) - h(n) \pm O(g(n)^2 + h(n)^2)).
		\end{equation}
	\end{lemma}
	\begin{proof}
		Recall from~\eqref{eq:pl}
		that the
		cumulative distribution function for the weights is
		$F(x)  =  1 - x^{1 - \tau}$.
		Thus, we get
		\begin{align}
			\Exp{|S|}
			\notag
			&= n \cdot \left(
			F\left((a - h(n))^b \sqrt{\mu n}\right) -
			F\left((a - g(n))^b \sqrt{\mu n}\right)\right)\\
			\notag
			&= n \cdot \left(
			((a - g(n))^b \sqrt{\mu n})^{1 -\tau}  - 
			((a - h(n))^b \sqrt{\mu n})^{1 - \tau} \right) \\
			\label{eq:nr-vertices-in-range-1}
			&= \mu^{\frac{1 - \tau}{2}}
			n^{\frac{3 - \tau}{2}}
			\left(
			(a - g(n))^{b(1 - \tau)} - (a - h(n))^{b(1 - \tau)}
			\right).
		\end{align}
		We can now use the Taylor expansion of $f(x) = (a - x)^c$ at $0$ to
		obtain the bound $f(x) = a^c - c a^{c - 1} x \pm O(x^2)$, which is
		valid for $x \in o(1)$.  Since $g(n), h(n) \in o(1)$ we can thus bound the
		above term in parentheses for $c = b(1 - \tau)$ as
		\begin{align}
			(a - g(n))^c - (a - h(n))^c
			\notag
			&=
			- c a^{c - 1} g(n) + c a^{c - 1} h \pm O(g(n)^2 + h(n)^2)\\
			\notag
			&= - c a^{c - 1}(g(n) - h(n) \pm O(g(n)^2 + h(n)^2))\\
			\label{eq:nr-vertices-in-range-2}
			&= b(\tau - 1)a^{b(1 - \tau) - 1} (g(n) - h(n) \pm O(g(n)^2 + h(n)^2)).
		\end{align}
		Equations~\eqref{eq:nr-vertices-in-range-1}
		and~\eqref{eq:nr-vertices-in-range-2} together yield the 
		claim.
	\end{proof}
	
	Note that, if additionally $h(n) \in o(g(n))$, we can write the last factor
	as $g(n)(1 - h(n)/g(n) \pm O(g(n))) = g(n)(1 \pm o(1))$ and obtain the
	following corollary.

	
	\begin{corollary}\label{col:nr-vertices-in-weight-range-dominant}
		Let $a, b > 0$ be constants and let $g(n), h(n)$ be functions of $n$ such
		that $g(n) \in o(1)$ and $h(n) \in o(g(n))$.  Let $S$ be the set of vertices
		with weight in $[(a - g(n))^b \sqrt{\mu n}, (a - h(n))^b \sqrt{\mu n})$.
		Then
		\begin{equation}\label{eq:expS}
			\Exp{|S|} =
			g(n) n^{\frac{3 - \tau}{2}} \cdot
			\mu^{\frac{1 - \tau}{2}} ba^{b(1 - \tau) - 1} \cdot
			(1 \pm o(1)).
		\end{equation}
	\end{corollary}
	
	Consider again the weights $w_\ell$ and $w_u$ as given in
	Lemma~\ref{lem:evenly-spaced-boxes-co-matching} and let $S$ be the set
	of vertices in $[w_\ell, w_u)$.  Then
	Corollary~\ref{col:nr-vertices-in-weight-range-dominant} in particular
	implies that $S$ contains $\Theta(g(n) \cdot n^{\frac{3 - \tau}{2}})$
	vertices in expectation.
	
	With this, we turn to our second goal mentioned above, namely that
	each box $B_i$ should be sufficiently large.
	
	\begin{lemma}
		\label{lem:volume-of-boxes}
		Let $B_1, \dots, B_{2k}$ be evenly spaced boxes of height $h(n)$ and
		gap $g(n)$ in $\mathbb T^d$.  If $g(n) \in o(1)$ then each box $B_i$ has
		volume $h / 2^{d - 1} \cdot (1 - o(1))$.
	\end{lemma}
	\begin{proof}
		Recall that the height of $B_i$ is $h(n)$ while its extent in all other
		dimensions is $\ell = \frac{1}{2} - g(n)$.  Thus its volume is
		$h(n) \cdot (\frac{1}{2} - g(n))^{d - 1} = h(n) / 2^{d - 1} \cdot (1 - 2g(n))^{d
			- 1}$.  The claim follows from the fact that $(1 - 2g(n))^{d - 1}$
		approaches $1$ from below for $n \to \infty$ as $g(n) \in o(1)$ and $d$
		constant.
	\end{proof}
	
	Corollary~\ref{col:nr-vertices-in-weight-range-dominant} and
	Lemma~\ref{lem:volume-of-boxes} together tell us that the expected
	number of vertices in each box that have a weight in the desired range
	is in $\Theta(h(n) \cdot g(n) \cdot n^{\frac{3 - \tau}{2}})$.  Recall we
	want to choose $h(n)$ and $g(n)$ as small as possible such that each box
	still contains a vertex with high probability.  We set
	$h(n) = c\cdot n^{-\frac{3 - \tau}{4}}$ and
	$g(n) = c\cdot n^{-\frac{3 - \tau}{4} + \varepsilon}$ for arbitrary
	constants $c > 0$ and $\varepsilon > 0$.  Note that this satisfies the
	condition $h(n) \in o(g(n))$ of
	Corollary~\ref{col:nr-vertices-in-weight-range-dominant} and yields an
	expected number of $\Theta(n^\varepsilon)$ vertices with the desired
	weight in each box.  Since the number of vertices in a given box
	follows a binomial distribution and since
	$n^{\varepsilon} \in \omega(\log(n))$, we can apply a Chernoff
	bound to conclude that actual number of vertices matches the expected
	value (up to constant factors) with probability $1 - O(n^{-c'})$ for
	any $c' > 0$~\cite[Corollaries 2.3 and 2.4]{bff-espsf-22}.  Together
	with a union bound, it follows that every box contains
	$\Theta(n^{\varepsilon})$ vertices (and thus at least one vertex) with
	probability $1 - O(2k \cdot n^{-c'})$.  By choosing $g(n), h(n)$, and $k$
	appropriately, we obtain the following theorem. 
	
	
	\begin{theorem}
		\label{thm:girg-torus-co-matching}
		Let $G^{(n)}$ be a $d$-dimensional GIRG with $T=0$, $\mu>0$ and $\tau\in(2,3)$ and let $s > 0$  and $\varepsilon > 0$ be
		arbitrary constants. Then, with high probability, $G^{(n)}$ contains a
		co-matching of size $s \cdot n^{\frac{3 - \tau}{4} - \varepsilon}$
		as induced subgraph.
	\end{theorem}
	\begin{proof}
		Let $B_1, \dots, B_{2k}$ be evenly spaced boxes of height
		$h(n) = c\cdot n^{-\frac{3 - \tau}{4}}$ and gap
		$g(n) = c\cdot n^{-\frac{3 - \tau}{4} + \varepsilon}$ (for
		appropriately chosen $c > 0$, which will be determined later).  Let
		$w_\ell$ and $w_u$ be defined as in
		Lemma~\ref{lem:evenly-spaced-boxes-co-matching}.  As argued above,
		Corollary~\ref{col:nr-vertices-in-weight-range-dominant} and
		Lemma~\ref{lem:volume-of-boxes} imply that, with high probability,
		each box $B_i$ includes at least one vertex with weight in
		$[w_\ell, w_u)$.  By
		Lemma~\ref{lem:evenly-spaced-boxes-co-matching} any set that contains exactly one vertex of each box
		forms a co-matching of size $2k$.
		
		Recall that $2k = 1 / (g(n) + h(n))$.  Thus, we can choose $c$ such that
		$2k = s \cdot n^{\frac{3 - \tau}{4} - \varepsilon}$.  Again, by the
		above argumentation, it follows that every box contains at least one
		vertex with probability
		$1 - O(2k n^{-c'}) = 1 - O(n^{\frac{3 - \tau}{4} - \varepsilon -
			c'})$ for any constant $c' > 0$.  Choosing $c'$ sufficiently large
		then yields the claim.
	\end{proof}
	
	This theorem together with Lemma~\ref{lem:co-matching-nr-cliques}
	directly imply the following corollary.
	
	\begin{corollary}\label{cor:maxcliquesgirg}
		Let $G^{(n)}$ be a $d$-dimensional GIRG with $T = 0$, $\mu>0$ and $\tau\in(2,3)$, and let $b > 0$ and $\varepsilon > 0$ be arbitrary
		constants.  Then, with high probability, the number of maximal
		cliques in $G^{(n)}$ is at least $b^{n^{(3 - \tau)/4 - \varepsilon}}$.
	\end{corollary}
	
	Figure~\ref{fig:girglb} shows this lower bound for $b=2$ against $n$. Interestingly, while Corollary~\ref{cor:maxcliquesgirg} shows that the number of maximal cliques grows super-polynomially in $n$, for $\tau> 2$, this growth may still be slower than the linear slope $n$ for large geometric networks. This is of particular importance as the smaller order terms of the number of maximal cliques contain terms of at least $\Theta(n)$. Indeed, the number of maximal 2-cliques is lower bounded by the number of vertices of degree 1, which scales linearly by Equation~\eqref{eq:pl}. Thus, for practical purposes, the dominant term could be the linear term instead of the super-polynomial term, especially if the degree exponent is close to 3. 
	
	\begin{figure}
		\centering
		\hfill
		\begin{subfigure}{0.32\linewidth}
			\includegraphics[width=\linewidth]{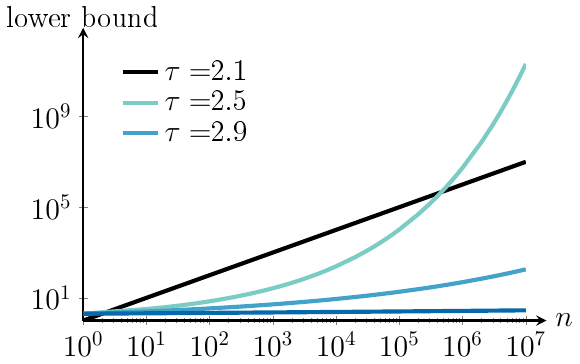}
			\caption{GIRG lower bound}
			\label{fig:girglb}
		\end{subfigure}
		\hfill
		\begin{subfigure}{0.32\linewidth}
			\includegraphics[width=\linewidth]{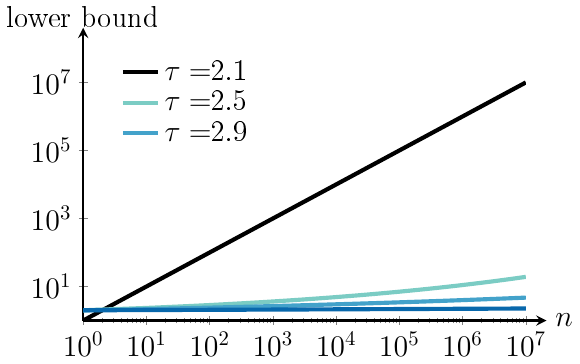}
			\caption{GIRG lower bound, no torus}
			\label{fig:girglbnotorus}
		\end{subfigure}
		\hfill\begin{subfigure}{0.32\linewidth}
			\includegraphics[width=\linewidth]{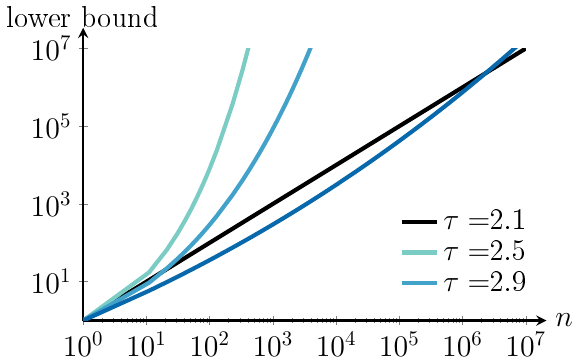}
			\caption{IRG lower bound}
			\label{fig:IRG lower bound}
		\end{subfigure}
		\hfill
		\caption{Lower bound on the number of maximal cliques of Corollary~\ref{cor:maxcliquesgirg} (with $b=2$, $\varepsilon\downarrow 0$), ~\ref{thm:girg_non_torus} (for $\varepsilon\downarrow 0$, $b=2$ and $C=1$), ~\ref{thm:max_cliques_irg} (for $\varepsilon\downarrow 0$ and $b=2$) against $n$ for different values of $\tau$. The black line is the line $n$.}
		\label{fig:plotexponents}
	\end{figure}
	
	\subsection{GIRG with 2-Dimensional Square}
	\label{sec:girg-with-2d-square}

	
	Our previous lower bound for the number of maximal cliques relies on the toroidal structure of the underlying space. We now show that even if the vertex positions are constrained to be positioned in the square $[0,1]^2$ instead, the GIRG still contains a super-polynomial number of maximal cliques. In this setting, we will also switch from the infinity norm to the 2-norm. We will discuss possible extensions to other norms in Section~\ref{sec:conc}.
	
	\begin{theorem}\label{thm:girg_non_torus}
		For any $\varepsilon>0$ and $b>0$, a 2-dimensional GIRG $G^{(n)}$ on $n$ vertices with vertex positions uniformly distributed over $[0,1]^2$ equipped with the 2-norm and $T=0$, $\mu>0$ and $\tau\in(2,3)$ contains with high probability at least 
		\begin{equation}
			Cb^{n^{\frac{3-\tau}{10}-\varepsilon}}
		\end{equation}
		maximal cliques for some $C>0$.
	\end{theorem}

 	\begin{figure}[t]
		\centering
		\begin{subfigure}{0.3\linewidth}
			\includegraphics[width=\linewidth]{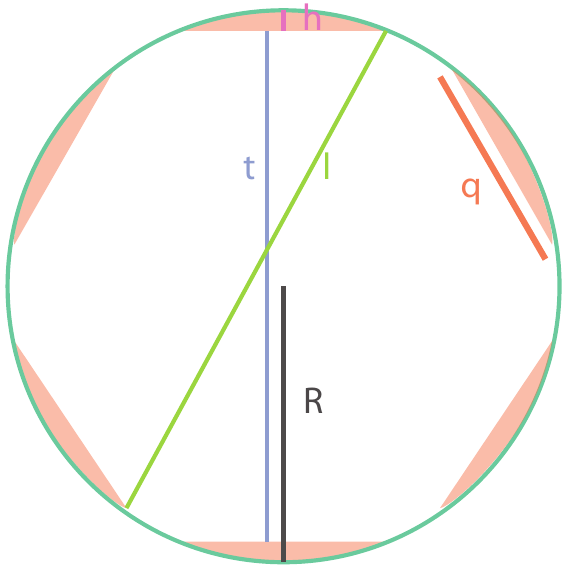}
			\caption{A circle of radius $R$ with several orange areas of height $h(n)$. }
			\label{fig:circle}
		\end{subfigure}
		\hspace{3cm}
		\begin{subfigure}{0.3\linewidth}
			\includegraphics[width=\linewidth]{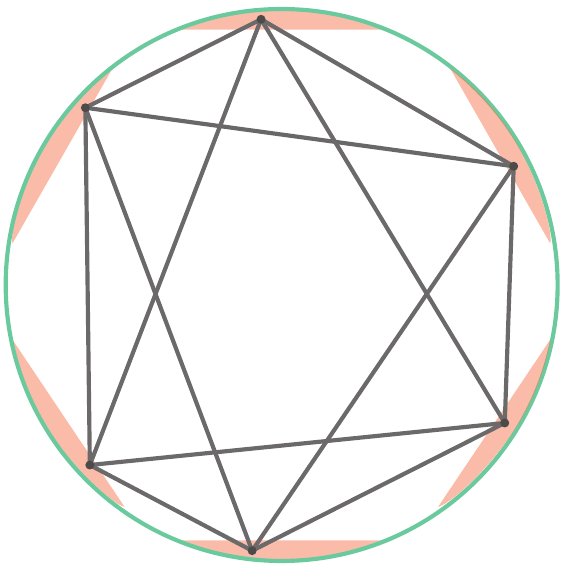}
			\caption{Any set of vertices with one in each orange area form a co-matching. }
			\label{fig:circleclique}
		\end{subfigure}
		\caption{Clique minus a matching in the 2-dimensional GIRG.}
		\label{fig:2dimgirg}
	\end{figure}
	\begin{proof}
		Let $S$ be the set of vertices with weights within $[a\sqrt{\mu n(1-c\cdot n^{-\beta})},a\sqrt{\mu n}]$ for some $0<a<1/4$, $\beta>0$ (and
		appropriately chosen $c > 0$, which will be determined later). By Corollary~\ref{col:nr-vertices-in-weight-range-dominant},
		\begin{equation}
			\Exp{|S|}
			=n^{(3-\tau)/2-\beta}(1+o(1))
		\end{equation}
		Let $\mathcal{C}$ be a circle on $[0,1]^2$ of constant radius $R<1/4$. We now create an even number of areas $B_1,\dots,B_{2k}$ of height $h(n)$, evenly distributed over $\mathcal{C}$ as illustrated in Figure~\ref{fig:circle}. That is, we consider $2k$ identical and evenly spaced circular segments $B_1,\dots, B_{2k}$ of height $h$ and chord length $q$. We ensure that any pair of vertices in two opposite areas $B_i$ and $B_{i+k}$ are disconnected. That is, the distance $t$ between the two ends of these areas should equal
		\begin{equation}
			t=a.
		\end{equation}
		We also ensure that any pair of vertices in non-opposite areas connect. This means that the distance $\ell$ between the rightmost part of $B_i$ and the leftmost part of any non-opposite $B_j$ is at most
		\begin{equation}\label{eq:l2dimnotorus}
			\ell=a\sqrt{1-c \cdot n^{-\beta}}.
		\end{equation}
		The width $q$ of one of the $B_i$'s is given by
		\begin{equation}
			\frac{q}{2}=\sqrt{R^2-(R-h)^2}=\sqrt{h(2R-h)}=\sqrt{h(n)(h(n)+t)}=\sqrt{h(n)(h(n)+a)}.
		\end{equation}
		Thus, the area of $B_i$ is given by
		\begin{equation}
			c_1h(n)^{3/2}\sqrt{h(n)+a},
		\end{equation}
		for some constant $c_1>0$. 
		The probability that a given area contains no vertices from $S$ is given by
		\begin{equation}\label{eq:areaempty}
			\Big(1-c_1h(n)^{3/2}\sqrt{h(n)+a}\Big)^{|S|}= \exp(-c_2n^{(3-\tau)/2-\beta}h(n)^{3/2})(1+o(1)),
		\end{equation}
		for some $c_2>0$.
		
		
		We now calculate the maximal number of areas that we can pack on $\mathcal{C}$. The circumference of $\mathcal{C}$ is $2\pi R$. The arc length of a single area is at most $q\pi/2$. 
		Furthermore, the arc length of the section with a chord of length $\ell$, $a(\ell)$, is given by 
		\begin{align}\label{eq:arc_length_l}
			a(\ell)= 2R\sin^{-1}(\ell/(2R)) & = 2R\sin^{-1}\Big(\frac{a\sqrt{1-c \cdot n^{-\beta}}}{a+2h(n)}\Big)
			\nonumber\\
			& = 2R\sin^{-1}\Big(1-\frac{a+2h(n)-a\sqrt{1-c\cdot n^{-\beta}}}{a+2h(n)}\Big)\nonumber\\
			& = \pi R- 2^{3/2}R\sqrt{\frac{a+2h(n)-a\sqrt{1-c\cdot n^{-\beta}}}{a+2h(n)}} (1+o(1)),
		\end{align}
		using the Taylor series of $\sin^{-1}(1-x)$ around $x=0$. Now take $h(n)=c\cdot n^{-\gamma(n)}$. Then, by Equation~\eqref{eq:arc_length_l}, $a(\ell)$ scales as
		\begin{align}
			a(\ell)= \pi R- R\Theta\Big(\sqrt{(1-\sqrt{1-c\cdot n^{-\beta}})+c\cdot n^{-\gamma(n)}}\Big)= \pi R-  R\Theta(\sqrt{c\cdot n^{-\beta}+c\cdot n^{-\gamma(n)}}).
		\end{align}
		Now the arc length between two adjacent sections $B_i$ and $B_{i+1}$ is equal to $\pi R-a(\ell)$. 
		This means that the arc length between $B_i$ and $B_{i+1}$ scales as $\pi R-(\pi R-R\Theta\sqrt{c\cdot n^{-\beta}+c\cdot n^{-\gamma(n)})})=R\Theta(\sqrt{c\cdot n^{-\beta}+c\cdot n^{-\gamma(n)}})$.

		
		The maximal value of the number of possible areas $2k$, is the total circumference  of $\mathcal{C}$ divided by the arc length of an interval $B_i$ and the arc length between $B_i$ and $B_{i+1}$, which yields 
		\begin{equation}\label{eq:mbound}
			2k= \frac{2\pi R}{R\Theta(\sqrt{c\cdot n^{-\beta}+c\cdot n^{-\gamma(n)}})+ \sqrt{h(n)(h(n)+a)}}\in \Theta\Big(\min(n^{\beta/2},n^{\gamma(n)/2})\Big) .
		\end{equation}
		Thus, by choosing $c$ correctly, we can let $2k= s\cdot  \min(n^{\beta/2},n^{\gamma(n)/2})$ for any $s>0$. When all $i\in[2k]$ contain at least one vertex in $S$, any set of $2k$ vertices with exactly one vertex in each $B_i$ forms a co-matching, as illustrated in Figure~\ref{fig:circleclique}. Furthermore~\eqref{eq:areaempty} shows that with high probability, all $B_i$ are non-empty, as long as $n^{(3-\tau)/2-\beta}h(n)^{3/2}\to\infty$ as $n\to\infty$. 
		
		We therefore choose $\beta=(3-\tau)/5-\varepsilon$ and $\gamma(n)=(3-\tau)/5$. Then, with high probability there is a co-matching of size $2k =s\cdot  n^{(3-\tau)/10-\varepsilon}$. Thus, by Lemma~\ref{lem:co-matching-nr-cliques} and choosing $s$ sufficiently large yields that for fixed $b>0$ the number of maximal cliques can be bounded from below by
		\begin{equation}
			b^{ n^{(3-\tau)/10-\varepsilon}}
		\end{equation}
	\end{proof}
	

	Figure~\ref{fig:girglbnotorus} shows the lower bound of Theorem~\ref{thm:girg_non_torus} against $n$. As for the toroidal case, the super-polynomial growth may be dominated by lower-order linear terms. 
	
	\subsection{Non-Threshold Case}
        \label{sec:girg-non-threshold-case}
	We now show how our constructions extend to the non-threshold
        GIRG, where the connection probability is given by
        Equation~\eqref{eq:edgeprob} instead of
        Equation~\eqref{eq:edgeprobthreshold}.
	\begin{theorem}
          \label{thm:non-threshold-torus}
          Let $G^{(n)}$ be a $d$-dimensional GIRG on $n$ vertices with $T > 0$, $\mu>0$ and $\tau\in(2,3)$ and let
          $s > 0$ be an arbitrary constant.  Then, there exists an
          $\varepsilon > 0$ such that, with high probability, $G$
          contains a co-matching of size
          $s \cdot n^{(3 - \tau)/5} \cdot (\varepsilon \log n)^{-(1/2)}$ as induced subgraph.
	\end{theorem}
	\begin{proof}
          As before, we consider $2k$ boxes $B_1, \ldots, B_{2k}$ with
          height $h(n)$ and gap $g(n)$, though now we choose
          \begin{align*}
            g(n) = h(n) \cdot (\varepsilon \log n)^{1/2} \qquad \text{and} \qquad h(n) = \frac{1}{2s} \cdot n^{-\frac{3 - \tau}{5}},
          \end{align*}
          for a constant $\varepsilon > 0$, which we determine below.
          We again focus on the vertex set $S$ containing all vertices
          with weights in $[w_\ell, w_u]$, though our choice for $w_u$
          is slightly different.  In particular, we choose
          \begin{align*}
            w_\ell = (1/2 - g(n))^{d/2} \sqrt{\mu n} \qquad \text{and} \qquad w_u = (1/2 - (T/d + 1)h)^{d/2}\sqrt{\mu n}.
          \end{align*}

          Our goal now is to show that, with high probability, there
          exists at least one co-matching that contains one vertex
          from each box.  That is, if $M$ denotes the number of such
          co-matchings, we want to show that $M > 0$ with high
          probability.

          We start by bounding the number of vertices from
          $S$ that lie in a given box $B_i$, denoted by $S(B_i)$.
          Since $T$ and $d$ are constants and
          $(\varepsilon \log n)^{1/2} \in \omega(1)$, we have
          $(T/d + 1)h(n) \in o(g(n))$, allowing us to bound
          $\mathbb{E}[|S|]$ using
          Corollary~\ref{col:nr-vertices-in-weight-range-dominant},
          which yields
          \begin{align*}
            \mathbb{E}[|S|] \in \Theta\left(gn^{\frac{3 - \tau}{2}}\right).
          \end{align*}
          Moreover, since the vertices are distributed uniformly at
          random in the ground space, the expected fraction of vertices from
          $S$ that lie in the box $B_i$ is proportional to its volume,
          which is $\Theta(h)$ according to
          Lemma~\ref{lem:volume-of-boxes}.  It follows that
          \begin{align*}
            \mathbb{E}[|S(B_i)|] \in \Theta\left(g h n^{\frac{3 - \tau}{2}}\right) \in \Theta\left(h(n)^2 n^{\frac{3 - \tau}{2}} (\varepsilon \log n)^{1/2}\right) \in \Theta\left( n^{(3 - \tau)(1/2 - 2/5)} (\varepsilon \log n)^{1/2}\right).
          \end{align*}
          Analogous to the proof of Lemma~\ref{lem:volume-of-boxes} we
          can apply a Chernoff bound to conclude that the
          number of vertices in $S(B_i)$ matches the expected value
          (up to constant factors) with probability $1 - O(n^{-c})$
          for any $c > 0$.  Note that the number of boxes is given by
          \begin{align}
            \label{eq:temperature-co-match-size}
            2k = \frac{1}{g(n) + h(n)} = \frac{1}{h((\varepsilon \log n)^{1/2} + 1)} = \frac{2s \cdot n^{\frac{3 - \tau}{5}}}{(\varepsilon \log n)^{1/2} + 1},
          \end{align}
          which is at most $n$. Thus, applying the union bound yields
          that with high probability \emph{every} box contains
          $n' \in \Theta(n^{(3 - \tau)(1/2 - 2/5)} (\varepsilon \log
          n)^{1/2})$ vertices.  In the following, we implicitly
          condition on this event to happen.  Now recall that a
          co-matching consisting of one vertex from each box forms if
          each vertex is adjacent to the vertices in all other boxes,
          \emph{except} the vertex from the opposite box.

          Despite the temperature, vertices in non-opposite boxes are
          still adjacent with probability $1$, since the weight of two
          such vertices $i$ and $j$ is at least $w_\ell$ and their
          distance is at most $\frac{1}{2} - g(n)$ and, thus, according
          to Equation~\ref{eq:edgeprob}
          \begin{align*}
            p_{ij} = \min \left\{ \left( \frac{w_i w_j}{n \mu ||x_i - x_j||_2^d}\right)^{1/T}, 1 \right\} \ge \min \left\{ \left(\frac{w_\ell^2}{n \mu (1/2 - g(n))^d}\right)^{1/T}, 1\right\} = 1.
          \end{align*}

          In contrast to the threshold case, however, the probability
          for vertices in opposite boxes to be adjacent is no longer
          $0$.  Since two such vertices $i$ and $j$ have distance at
          least $\frac{1}{2} - h(n)$ and weight at most
          $w_u = (\frac{1}{2} - (T/d + 1)h)^{d/2}\sqrt{\mu n}$, we can
          bound the probability for them to be adjacent using
          Equation~\ref{eq:edgeprob}, which yields
          \begin{align*}
            p_{ij} &= \min \left\{\left(\frac{w_i w_j}{n \mu ||x_i - x_j||_2^d}\right)^{1/T}, 1\right\} \\
                   & \le \min \left\{ \left(\frac{w_u^2}{n \mu (\frac{1}{2} - h(n))^d}\right)^{1/T} , 1\right\} \\
                   &= \min \left\{ \left(\frac{1 - (T/d + 1)2h(n)}{1 - 2h(n)}\right)^{d/T} , 1\right\} \\
                   &= \min \left\{ \left(1 - \frac{T}{d} \cdot \frac{2h(n)}{1 - 2h(n)}\right)^{d/T} , 1\right\} \\
                   &\le \min \left\{ \left(1 - \frac{T}{d} \cdot 2h(n) \right)^{d/T} , 1\right\}.
          \end{align*}
          Since $1 - x \le e^{-x}$, we obtain $p_{ij} \le e^{-2h(n)}$.

          With this we are now ready to bound the probability
          $\Prob{M > 0}$, that at least one co-matching forms that
          contains one vertex from each box.  To this end, we need to
          find one non-edge in each pair of opposite boxes, i.e., each
          such pair needs to contain two vertices (one from each box)
          that are not adjacent.  Conversely, the only way to
          \emph{not} find a co-matching is if there exists one pair of
          opposite boxes such that all vertices in one box are
          adjacent to all vertices in the other.  This happens with
          probability at most $(p_{ij})^{(n')^2}$.  Since there are
          $k$ pairs of opposite boxes, applying the union bound yields
          \begin{align}\label{eq:pm0}
            \Prob{M = 0} \le k(p_{ij})^{(n')^2} \le k \exp(-2h(n) (n')^2).
          \end{align}
          Now recall that
          $n' \in \Theta(n^{(3 - \tau)(1/2 - 2/5)} (\varepsilon \log n)^{1/2})$ and
          that $h(n) = 1/(2s) \cdot n^{-(3 - \tau)/5}$.  Consequently, we obtain
          \begin{align*}
            \Prob{M = 0} &\le k \exp\left(- \Theta\left(n^{-(3-\tau)/5} \cdot n^{(3 - \tau)(1 - 4/5)} \cdot \varepsilon \log n \right) \right) \\
                         &= k \exp\left(-\Theta(\varepsilon \log n)\right) \\
                         &= k n^{-\Theta(\varepsilon)}. \\
          \end{align*}
          Moreover, since $k \in O(n^{(3 - \tau)/5})$ (see
          Equation~\ref{eq:temperature-co-match-size}), we have
          \begin{align*}
            \Prob{M = 0} \in O \left( n^{(3 - \tau)/5 - \Theta(\varepsilon)} \right),
          \end{align*}
          meaning, for sufficiently large $n$, we can choose
          $\varepsilon$ such that $\Prob{M = 0} \in O(n^{-1})$ and,
          conversely, $\Prob{M > 0} = 1 - O(n^{-1})$.  So with high
          probability there exists at least one co-matching of size
          \begin{align*}
            2k = \frac{2s \cdot n^{\frac{3 - \tau}{5}}}{(\varepsilon \log n)^{1/2} + 1} \ge \frac{2s \cdot n^{\frac{3 - \tau}{5}}}{2 (\varepsilon \log n)^{1/2}} = \frac{s \cdot n^{\frac{3 - \tau}{5}}}{(\varepsilon \log n)^{1/2}},
          \end{align*}
          where the inequality holds for sufficiently large $n$.
        \end{proof}

	Together with Lemma~\ref{lem:co-matching-nr-cliques} we obtain
        the following corollary.
	
	\begin{corollary}\label{cor:maxcliquesgirg-non-threshold}
          Let $G^{(n)}$ be a $d$-dimensional GIRG on $n$ vertices with $T > 0$, $\mu>0$ and $\tau\in(2,3)$ and let
          $b > 0$ be an arbitrary constant.  Then there exists an
          $\varepsilon > 0$ such that, with high probability, the
          number of maximal cliques in $G^{(n)}$ is at least
          \begin{equation}
              b^{n^{(3 - \tau)/5} \cdot (\varepsilon \log n)^{-(1/2)}}.
          \end{equation}
	\end{corollary}
	
	We can extend Theorem~\ref{thm:girg_non_torus} to non-zero temperature in a very similar fashion (proof is in Appendix~\ref{app:prooftemperature})
	\begin{theorem}\label{thmnonzero2dim}
		For any $\varepsilon>0$ and $b>1$, a 2-dimensional GIRG $G^{(n)}$ on $n$ vertices with $T>0$ and vertex positions uniformly distributed over $[0,1]^2$ contains with high probability at least
		\begin{equation}
			b^{n^{\frac{3-\tau}{12}-\varepsilon}\log(n)^{1-\varepsilon}}
		\end{equation}
		maximal cliques. 
	\end{theorem}


\section{Inhomogeneous Random Graphs (IRG)}\label{sec:IRG}
We now turn to a random graph model that is scale-free, but does not
contain a source of geometry, the inhomogeneous random graph (IRG), or Chung-Lu random graph~\cite{chung2002}. We
show that also in this model, the number of maximal cliques scales
super-polynomially in the network size $n$. Again, every vertex $i$ draws its weight $w_i$ independently from
the power-law distribution~\eqref{eq:pl}, where we will again assume that $\tau\in(2,3)$. Then, all pairs of vertices $u$ and $v$ connect independently with probability
\begin{equation}\label{eq:conprob}
  p(w_u, w_v) = \min\Big(\frac{w_u w_v}{\mu n},1\Big),
\end{equation}
where $\mu$ controls the expected average degree.  

To show a lower bound on the number of maximal cliques, we make use of
the fact that an IRG contains a not too small rather dense subgraph
with high probability.  The following theorem is obtained by looking
just at the subgraph induced by vertices with weights in a certain
range.  We chose the specific range to satisfy three criteria.  First,
the range is sufficiently large, such that the subgraph contains many
vertices.  Second, the range is sufficiently small such that all
vertex pairs in the subgraph are connected with a similar probability.
And third, the weights are large enough such that a densely connected
subgraph forms, but not so large that the vertices merge into a single
clique.

\begin{theorem}
  \label{thm:max_cliques_irg}
  Let $G^{(n)}$ be an IRG on $n$ vertices with $\tau \in (2, 3)$ and $\mu>0$ and let $b > 1$ and
  $\varepsilon \in (0, \frac{3 - \tau}{4})$ be arbitrary constants.
  Then, the expected number of maximal cliques in $G^{(n)}$ is in
  \begin{equation}
      \Omega\big(b^{n^{(3 - \tau) / 4 - \varepsilon} \log n}\big).
  \end{equation}
\end{theorem}
\begin{proof}
  We show that already the subgraph $G'$ induced by the vertices in a
  certain weight range has the claimed expected number of maximal cliques.  To
  define $G'$, we consider weights in $[w_\ell, w_u]$ with
  $w_\ell = \sqrt{(1 - g(n)) \mu n}$ and $w_u = \sqrt{(1 - h(n)) \mu n}$.
  To abbreviate notation, let
  \begin{equation*}
    \gamma(n) = n^{\frac{3 - \tau}{4}}.
  \end{equation*}
  For constants $a$ and $c$ we determine later, we choose $g(n)$ and $h(n)$
  as
  \begin{equation*}
    g(n) = a h(n) \quad \text{and} \quad
    h(n) = c n^\varepsilon \log(n) \gamma(n)^{-1}.
  \end{equation*}
  Note that $w_\ell < w_u$ if and only if $a > 1$.  Let $n'$ be the
  number of vertices in $G'$.  From
  Lemma~\ref{lem:nr-vertices-in-weight-range} it follows
  \begin{equation*}
    \Exp{n'}
    \in \Theta(\gamma(n)^2 \cdot n^\varepsilon \log(n) \gamma(n)^{-1})
    \in \Theta(n^\varepsilon \gamma(n) \log n).
  \end{equation*}
  As every vertex has independently the same probability to be in
  $G'$, a Chernoff bound implies that
  $n' \in \Theta(n^\varepsilon\gamma(n)  \log n)$ holds with high
  probability.  Thus, in the following, we implicitly condition on
  this event to happen.

  To give a lower bound on the number of maximal cliques in $G'$, we
  only count the number $N_k$ of maximal cliques of size $k$ with
  \begin{equation*}
    k = \frac{3 \varepsilon}{c} n^{-\varepsilon}\gamma(n).
  \end{equation*}
  We note that this is the same constant $c$ as in the definition of
  $h(n)$ above. As the number of maximal cliques in $G'$ is a lower bound for the number of maximal cliques in $G^{(n)}$, we lower bound the expectation of $N_k$, by the expected number of maximal cliques in $G'$,
  \begin{align}\label{eq:ENK}
    \Exp{N_k} \ge {n' \choose k} \Prob{C \text{ is a clique}} \Prob{C \text{ maximal} \mid C \text{ is a clique}}.
  \end{align}
  In the following, we give estimates for the three terms
  individually.

  We start with the event that $C$ is a clique.  Due to the lower and
  upper bound on the weights in $G'$, it follows that any pair of
  vertices in $G'$ is connected with probability at least
  $p_\ell = 1 - g(n)$ and at most $p_u = 1 - h(n)$.  Thus, for a fixed
  subset $C$ of vertices of size $|C| = k$, the probability that all
  $k$ vertices are pairwise connected is at least
  $p_\ell^{k (k - 1) / 2} = (1 - g(n))^{k (k - 1) / 2}$.  As $g(n)$ goes to
  $0$ for growing $n$ and $1 - x \in \Omega(\exp(-x))$ in this case,
  we get
  \begin{equation}
    \label{eq:irg-subset-is-clique}
    \Prob{C \text{ is a clique}}
    \ge (1 - g(n))^{k (k - 1) / 2}
    \in \Omega\left(\exp\left(- \frac{g k (k - 1)}{2}\right)\right).
  \end{equation}

  For $C$ to be a maximal clique (conditioning on it being a clique),
  additionally no other vertex can be connected to all vertices from
  $C$.  This probability is at least
  $(1 - p_u^k)^{n' - k} = (1 - (1 - h(n))^k)^{n' - k}$.  As
  $1 - x \le \exp(-x)$, it follows that
  $(1 - h(n))^k \le \exp(-hk) = n^{-3\varepsilon}$, where the last
  equality follows from plugging in the values we chose for $h(n)$ and
  $k$.  Again using $1 - x \in \Omega(\exp(-x))$ for
  sufficiently small $x$, we can conclude that
  \begin{align}
    \notag
    \Prob{C \text{ maximal} \mid C \text{ is a clique}}
    &\ge (1 - (1 - h(n))^k)^{n' - k}\\
    \label{eq:irg-clique-is-maximal}
    &\ge (1 - n^{-3\varepsilon})^{n' - k}\\
    \notag
    &\in \Omega\Big( \exp\big( -n^{-3\varepsilon} (n' - k) \big) \Big).
  \end{align}

  Finally, for the binomial coefficient, we get
  \begin{equation}
    \label{eq:irg-binomial-coefficient}
    {n' \choose k}
    \ge \left( \frac{n'}{k} \right)^k
    = \exp\left(\log\left(\frac{n'}{k}\right) k\right).
  \end{equation}

  Our goal is to show that $\log(\Exp{N_k})\geq n^{(3-\tau)/4-\varepsilon}\log(n)\log(b)$. Thus, plugging
  Equations~\eqref{eq:irg-subset-is-clique},~\eqref{eq:irg-clique-is-maximal},
  and~\eqref{eq:irg-binomial-coefficient} into the logarithm of~\eqref{eq:ENK} yields that we need to show that for
  every constant $b > 1$, we can choose the constants $a > 1$ and $c$
  in the definitions of $g(n)$ and $h(n)$ such that
  \begin{equation}\label{eq:claimconstants}
    \log\left(\frac{n'}{k}\right) k - \frac{g k (k - 1)}{2} - n^{-3\varepsilon} (n' - k)
    \stackrel{\text{\tiny(to be shown)}}{>} n^{-\varepsilon} \gamma(n) \log n \log b
    = n^{\frac{3 - \tau}{4} - \varepsilon} \log n \log b.
  \end{equation}
  This can be achieved by simply plugging in the values for $n'$, $k$,
  and $g(n)$.  For the first (and only positive) term, we obtain
  \begin{align*}
    \log\left(\frac{n'}{k}\right) k
    &= \log\left(\frac{\Theta(n^\varepsilon \log(n)
        \gamma(n))}{3\varepsilon / c n^{-\varepsilon}\gamma(n)}\right)
      \frac{3 \varepsilon}{c} n^{-\varepsilon}\gamma(n)\\
    &= \log\left(n^{2\varepsilon}\Theta(\log n)\right)
      \frac{3 \varepsilon}{c} n^{-\varepsilon}\gamma(n)\\
    \intertext{and thus for sufficiently large $n$}
    &\ge \frac{6 \varepsilon^2}{c} n^{-\varepsilon} \gamma(n) \log n.
  \end{align*}
  For the negative terms, we start with the latter and obtain
  \begin{equation*}
    n^{-3\varepsilon} (n' - k)
    \in \Theta(n^{-3\varepsilon} n^\varepsilon \gamma(n) \log n)
    =  \Theta(n^{-2\varepsilon} \gamma(n) \log n).
  \end{equation*}
  This is asymptotically smaller than the positive term and can thus
  be ignored.
  For the other negative term, first note that
  $gk = 3 a \varepsilon \log n$.  Thus, we obtain
  \begin{equation*}
    \frac{g k (k - 1)}{2}
    \le \frac{ 3 a \varepsilon }{2} \log(n) k
    = \frac{ 3 a \varepsilon}{2} \log(n) \frac{3 \varepsilon}{c} n^{-\varepsilon}\gamma(n)
    = \frac{ 9 a \varepsilon^2}{2 c} n^{-\varepsilon} \gamma(n) \log n.
  \end{equation*}
  Together with the positive term, we obtain that for sufficiently
  large $n$, it holds
  \begin{align*}
    \log\left(\frac{n'}{k}\right) k - \frac{g k (k - 1)}{2}
    &\ge \frac{6 \varepsilon^2}{c} n^{-\varepsilon} \gamma(n) \log n -
      \frac{ 9 a \varepsilon^2}{2 c} n^{-\varepsilon} \gamma(n) \log n\\
    &= \left(6  - \frac{9 a}{2}\right) \frac{\varepsilon^2}{c}
      n^{-\varepsilon} \gamma(n) \log n.
  \end{align*}
  With this, we can choose $a > 1$ such that the first factor is
  positive and we can choose $c$ such that
  $\varepsilon^2 / c = \log b$, which proves~\eqref{eq:claimconstants}, as then
  \begin{equation*}
      \left(6  - \frac{9 a}{2}\right) \frac{\varepsilon^2}{c}
      n^{-\varepsilon} \gamma(n) \log n =C n^{(3-\tau)/4-\varepsilon}\log(n)\log(b).
  \end{equation*}
  for some $C>0$.
\end{proof}

Figure~\ref{fig:IRG lower bound} shows that the lower bound provided by Theorem~\ref{thm:max_cliques_irg} may still be smaller than linear for networks that are quite large, especially when $\tau\approx 3$.
	
	\subsection{Small Maximal Cliques are Rare}
	We now focus on the maximal cliques of a fixed size in the IRG. How many maximal cliques of size $k$ are present in an IRG?

	Let $N(K_k)$ denote the number of maximal cliques of size $k$. Furthermore, let $M_n(\varepsilon)$ denote 
	\begin{equation}\label{eq:Mn}
		M_n(\varepsilon)=\{ ({v_1,\ldots, v_k})\colon  w_i\in[\varepsilon,1/\varepsilon] (\mu n) ^{\frac{\tau-2}{\tau-1}}\ \text{ for } i=1,2 \text{ and }  w_i\in[\varepsilon,1/\varepsilon] (\mu n) ^{\frac{1}{\tau-1}} \ \forall i\in\{3,\dots,k\} \}.
	\end{equation}
	Thus, $M_n(\varepsilon)$ is the set of sets of $k$ vertices
        such that two vertices have weight proportional to
        $n^{(\tau-2)/(\tau-1)}$, and all other vertices have weights
        proportional to $n^{1/(\tau-1)}$.  Denote the number of maximal
        $k$-cliques with sets of vertices in $M_n(\varepsilon)$ by
        $N(K_k,M_n(\varepsilon))$. Then, the following theorem shows
        that these `typical' maximal cliques are asymptotically all
        maximal cliques. Furthermore, it shows that all maximal
        cliques of size $k>2$ occur equally frequently in scaling, and
        they also appear on the same types of vertices.  Here we use
        $\plim$ to denote convergence in probability.
	
	\begin{theorem}[Maximal clique localization]\label{thm:maxcliqeslocalized}
		Let $G^{(n)}$ be an IRG on $n$ vertices with $\tau \in (2, 3)$ and $\mu>0$. For any fixed $k\geq 3$,
		\begin{enumerate}[(i)]
			\item 
			For any $\varepsilon_n$ such that $\lim_{n\to\infty}\varepsilon_n=0$,
			\begin{equation}
				\frac{N\big(K_k,M_n\left(\varepsilon_n\right)\big) }{N(K_k)}\plim 1.
			\end{equation}
			\item Furthermore, for any fixed $0<\varepsilon<1$,\
			\begin{equation}\label{eq:Nsubmag}
				\Exp{N(K_k,M_n(\varepsilon))}=\Theta(n^{(3-\tau)(2\tau-3)/(\tau-1)}).
			\end{equation}
		\end{enumerate}
	\end{theorem}
	Theorem~\ref{thm:maxcliqeslocalized}(i) states that asymptotically all maximal $k$-cliques are formed between two vertices of weights proportional to $n^{(\tau-2)/(\tau-1)}$ and all other vertices of weights proportional to $n^{1/(\tau-1)}$. 
 Theorem~\ref{thm:maxcliqeslocalized}(ii) then shows that there are proportional to $n^{(3-\tau)(2\tau-3)/(\tau-1)}$ such maximal $k$-cliques.
        As visualized in Figure~\ref{fig:madfunct}, this scaling is significantly smaller than the scaling of the total number of $k$-cliques, which scales as $n^{k/2(3-\tau)}$~\cite{hofstad2017d}. Interestingly, the scaling of the number of maximal cliques is $k$-independent, contrary to the total number of cliques.  In particular, the number of $k$ maximal cliques is always $o(n)$, contrary to the number of $k$-cliques which scales larger than $n$ when $\tau<3-2/k$. This shows once more that the large number of maximal cliques in the IRG is caused by extremely large maximal cliques, as fixed-size maximal cliques are only linearly many. 

 To prove this theorem, we need the following technical lemma, which is proven in Appendix~\ref{app:proofintfinite}:
 \begin{lemma}\label{lem:intfinite}
 When $\tau\in(2,3)$, then
		\begin{equation}\label{eq:intmaxclique}
			\int_0^1\dots\int_0^1 x_3^{1-\tau}\cdots x_k^{1-\tau} \int_0^\infty\int_{x_1}^\infty x_1^{k-1-\tau}x_2^{1-\tau}\prod_{i=3}^k\min\Big(x_2x_i,1\Big)e^{-\mu^{1-\tau}x_1x_2^{\tau-2}}dx_2 dx_1\dots dx_k<\infty.
		\end{equation}
	\end{lemma}

Furthermore, we need a lemma that bounds the probability that a given clique on vertices of weights $x_1\leq x_2\dots\leq x_k$ is maximal:
	\begin{lemma}\label{lem:pmaximal}
		Let $G$ be an IRG with $\tau \in (2, 3)$ and $\mu>0$. Then, the probability that a given clique between $k$ vertices of weights $x_1\leq x_2\dots\leq x_k$ is maximal is bounded by
			\begin{align}
			 &\quad  \exp\Big(-C_1n^{2-\tau}\mu^{1-\tau}x_{1}x_{2}^{\tau-2}\Big)(1+o(1)) \leq \Prob{\text{clique on weights }x_1,\dots,x_k\text{ maximal}}\nonumber\\
    & \leq \exp\Big(-C_2n^{2-\tau}\mu^{1-\tau}x_{1}x_{2}^{\tau-2}\Big),
		\end{align}
	for some $0<C_1\leq C_2<\infty$.
	\end{lemma}
\begin{proof}
	When $x_{1}\leq x_2\leq \dots \leq x_k$, we can compute the probability that this $k$ clique is part of a larger clique with a randomly chosen vertex as
	\begin{align}
		& \int_1^{\infty}w^{-\tau}\prod_{i\in[k]}\min\Big(\frac{w x_{i}}{\mu n},1\Big)dw \nonumber\\
		& = \frac{x_{1}\dots x_k}{(\mu n)^k}\int_1^{\mu n/ x_k}w^{k-\tau}dw + \frac{x_1\dots x_{{k-1}}}{(\mu n)^{k-1}}\int_{\mu n/ x_k}^{\mu n/ x_{{k-1}}}w^{k-1-\tau}dw\nonumber\\
		& \quad + \dots + \frac{x_1x_{{2}}}{(\mu n)^{2}}\int_{\mu n/ x_{3}}^{\mu n/ x_{{2}}}w^{2-\tau}dw + \frac{x_1}{\mu n}\int_{\mu n/ x_2}^{\mu n/ x_{{1}}}w^{1-\tau}dw+ \int_{\mu n/ x_1}^{\infty}w^{-\tau}dw\nonumber\\
		& =  c_k\frac{x_1\dots x_k}{(\mu n)^k}\Big(\frac{\mu n}{x_k}\Big)^{k+1-\tau}+ \dots 
		+  c_2\frac{x_1}{\mu n}\Big(\frac{\mu n}{x_2}\Big)^{2-\tau}+c_1\Big(\frac{\mu n}{x_1}\Big)^{1-\tau},
	\end{align}
	for some $c_1,\dots, c_k>0.$
	When $x_1\leq x_2\leq \dots\leq  x_k$, this term becomes
	\begin{equation}
		(\mu n)^{1-\tau}\sum_{l=1}^kc_lx_{l}^{-l+\tau}\prod_{i<l}x_{i}.
	\end{equation}
	The ratio between two consecutive terms of this summation equals
	\begin{equation}
		\frac{x_{l}^{\tau-l}x_1\dots x_{{l-1}}}{x_{{l+1}}^{\tau-l-1}x_1\dots x_{{l}}}=\Big(\frac{x_{l}}{x_{{l+1}}}\Big)^{\tau-l-1}.
	\end{equation}
	Now as $x_{l}\leq x_{{l+1}}$ and $\tau\in(2,3)$, this ratio is larger than 1 for $l\geq 2$, and smaller than one for $l=1$. This means that the summation can be dominated by 
	\begin{equation}\label{eq:ubsum}
		(\mu n)^{1-\tau}\sum_{l=1}^kc_lx_{l}^{-l+\tau}\prod_{i<l}x_{i}\leq C (\mu n)^{1-\tau}x_1x_2^{\tau-2},
	\end{equation}
	for some $C>0$. 
	
	Thus, the probability that a clique on vertices with weights $x_1,\dots, x_k$ is maximal can be upper bounded by
	\begin{align}
		\Prob{(x_1,\dots,x_k) \text{ clique maximal }} & \leq \Big(1-C (\mu n)^{1-\tau}x_1x_2^{\tau-2}\Big)^{n}\nonumber\\
		& \leq  \exp\Big(-Cn^{2-\tau}\mu^{1-\tau}x_1x_2^{\tau-2}\Big).
	\end{align}

We lower bound the probability that the clique is maximal by using that 
\begin{equation}
	(\mu n)^{1-\tau}\sum_{l=1}^kc_lx_{l}^{-l+\tau}\prod_{i<l}x_{i}\geq c_2 (\mu n)^{1-\tau}x_1x_2^{\tau-2}.
\end{equation}
Thus,
	\begin{align}
	\Prob{(x_1,\dots,x_k) \text{ clique maximal }} & \geq \Big(1-c_2 (\mu n)^{1-\tau}x_1x_2^{\tau-2}\Big)^{n}\nonumber\\
	&  \geq  \exp\Big(-c_2n^{2-\tau}\mu^{1-\tau}x_1x_2^{\tau-2}/(1+c_2n^{1-\tau}\mu^{1-\tau}x_1x_2^{\tau-2})\Big)\nonumber\\
	& = \exp\Big(-c_2n^{2-\tau}\mu^{1-\tau}x_1x_2^{\tau-2}\Big)(1+o(1)).
\end{align}
\end{proof}

	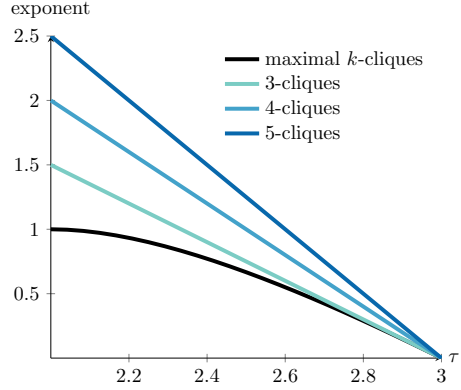
\begin{figure}
		\centering
  \scalebox{0.75}{
		\begin{tikzpicture}
			\definecolor{mycolor1}{rgb}{0.48235,0.80000,0.76863}%
			\definecolor{mycolor2}{rgb}{0.26275,0.63529,0.79216}%
			\definecolor{mycolor3}{rgb}{0.03137,0.40784,0.67451}%
			\begin{axis}[
				axis lines=center,
every axis x label/.style={at={(current axis.right of origin)},anchor=west},
every axis y label/.style={at={(current axis.north west)},above=2mm},
				xlabel=$\tau$,
    ylabel = exponent, 
    legend style={legend cell align=left, align=left, draw=none}]
				
				\addplot [
				domain=2:3, line width=2.0 pt, 
				] {(3-x)*(2*x-3)/(x-1)};
				\addlegendentry{maximal $k$-cliques};
				\addplot [
				domain=2:3, line width=2.0 pt,  color=mycolor1
				] {3/2*(3-x)}; 
				\addlegendentry{3-cliques};
				\addplot [
				domain=2:3, line width=2.0 pt, color=mycolor2
				] {4/2*(3-x)}; 
				\addlegendentry{4-cliques};
				\addplot [
				domain=2:3, line width=2.0 pt, color=mycolor3
				] {5/2*(3-x)}; 
				\addlegendentry{5-cliques};
			\end{axis}
		\end{tikzpicture}}
		\caption{Scaling of the number of maximal $k$-cliques, and the total number of (not necessarily maximal) 3,4,5-cliques.}
		\label{fig:madfunct}
	\end{figure}
	


 Now we are ready to prove Theorem~\ref{thm:maxcliqeslocalized}:

	\begin{proof}[Proof of Theorem~\ref{thm:maxcliqeslocalized}]
		Fix $\ell_i\leq u_i$ for $i\in[k]$. We now compute the expected number of maximal $k$-cliques in which the vertices have weights $n^{(\tau-2)/(\tau-1)}[\ell_i,u_i]$ for $i=1,2$, and $n^{1/(\tau-1)}[\ell_i,u_i]$ for $i\geq 3$.
%

		We bound the expected number of such maximal copies of $K_k$ by
		\begin{equation}
			\label{eq:Exp1small}
			\begin{aligned}[b]
				&\sum_{\boldsymbol{v}}\Exp{I(K_k, \boldsymbol{v})\ind{w_{v_i}\in [\ell_i,u_i]n^{(\tau-2)/(\tau-1)}, \ i =1,2, \ w_{v_i}\in [\ell_i,u_i] n^{1/(\tau-1)},\  i \geq 3 }}\\
				& = n^k\int_{\ell_1 n^{(\tau-2)/(\tau-1)}}^{u_1 n^{(\tau-2)/(\tau-1)}}\int_{\ell_2 n^{(\tau-2)/(\tau-1)}}^{u_2 n^{(\tau-2)/(\tau-1)}}\cdots \int_{\ell_k n^{1/(\tau-1)}}^{u_k n^{1/(\tau-1)}}(x_1\cdots x_k)^{-\tau}
				\prod_{\mathclap{1\leq i<j\leq k}}\min\left(\frac{x_ix_j}{n},1\right) \nonumber\\
    & \quad \cdot \Prob{(x_1,\dots,x_k) \text{ clique maximal}}\dd x_k\cdots\dd x_1,
			\end{aligned}
		\end{equation}
		where $I(K_k, \boldsymbol{v})$ is the indicator that a maximal $k$-clique is present on vertices $\boldsymbol{v}$, and the sum over $\boldsymbol{v}$ is over all possible sets of $k$ vertices.  Now the probability that a clique is maximal can be upper bounded as in Lemma~\ref{lem:pmaximal}.

		We bound the minimum in~\eqref{eq:Exp1small} by 
		\begin{itemize}
			\item[(a)] $x_ix_j/n$ for $\{i,j\}=\{1,2\}$ or $i=1,j\geq 3$;
			\item[(b)] 1 for $i,j\geq 3$ .
		\end{itemize}

		Making the change of variables $x_i=y_in^{1/(\tau-1)}$ for $i=3,\dots,k$ and $x_i=y_i/n^{(\tau-2)/(\tau-1)}$ otherwise, we obtain
		the bound
		\begin{align}\label{eq:expnhsmall}
				&\sum_{\boldsymbol{v}}\Exp{I(K_k, \boldsymbol{v})\ind{w_{v_i}\in [\ell_i,u_i]n^{(\tau-2)/(\tau-1)}, \ i =1,2, \ w_{v_i}\in [\ell_i,u_i] n^{1/(\tau-1)}, \ i \geq 3 }}\nonumber\\
				& \leq \tilde{K} n^{k}n^{2(\tau-2)/(\tau-1)-k+1}\nonumber\\
				& \times
				\int_{\ell_1}^{u_1}\int_{y_1}^{u_2}\int^{u_3}_{\ell_3}\cdots \int^{u_k}_{\ell_k}y_1^{2-\tau}y_2^{1-\tau}y_3^{1-\tau}\dots y_k^{1-\tau}
				\prod_{j\geq 3}\min(y_2y_j,1)\exp(-\mu^{1-\tau}y_1y_2^{\tau-2})\dd y_{k}\cdots \dd y_{1} ,
		\end{align}
  for some $\tilde{K}>0$.
  Because the weights are sampled i.i.d. from a power-law distribution, the maximal weight $w_{\max}$ satisfies that for any $\eta_n\to 0$,  $w_{\max}\leq n^{1/(\tau-1)}/\eta_n$ with high probability. 
	Thus, we may assume that $u_i\leq 1/\eta_n$ when $i\geq 3$.
	Now suppose that at least one vertex has weight smaller than $\varepsilon_n n^{(\tau-2)/(\tau-1)}$ for $i=1,2$ or smaller than $\varepsilon_n n^{1/(\tau-1)}$ for $i\geq 3$. This corresponds to taking $u_i=\varepsilon_n$ and $\ell_i=0$ for at least one $i$, or at least one integral in~\eqref{eq:expnhsmall} with interval $[0,\varepsilon_n]$. Similarly, when vertex 1 or 2 has weight higher than $1/\varepsilon_n n^{(\tau-2)/(\tau-1)}$, this corresponds to taking $\ell_i=1/\varepsilon_n$ and $u_i=\infty$ for $i=1$ or 2, or at least one integral in~\eqref{eq:expnhsmall} with interval $[1/\varepsilon_n,\infty]$.
	Lemma~\ref{lem:intfinite} then shows that these integrals tends to zero when choosing $u_i=\eta_n$ fixed for $i\geq 3$ and $\varepsilon_n\to 0$.
	 Thus, choosing $\eta_n\to 0$ sufficiently slowly compared to $\varepsilon_n$ yields that
		\begin{align}\label{eq:expcontrsmall}
			\sum_{\boldsymbol{v}}\Exp{I(K_k, \boldsymbol{v})\ind{\boldsymbol{v}\notin \Gamma_n(\varepsilon_n,\eta_n)}} \in o((n^{(3-\tau)(2\tau-3)/(\tau-1)}),
		\end{align}
		where
		\begin{equation}
			\Gamma_n(\varepsilon_n,\eta_n) = \{(v_1,\dots,v_k)\colon w_{v_i}\in n^{(\tau-2)/(\tau-1)}[\varepsilon_n,1/\varepsilon_n], i=1,2\ n^{1/(\tau-1)}[\varepsilon_n,1/\eta_n]\}.
		\end{equation}	
	
		Let $\bar{\Gamma}_n(\varepsilon_n,\eta_n)$ be the complement of $\Gamma_n(\varepsilon_n,\eta_n)$. Denote the number of maximal cliques with vertices in $\bar{\Gamma}_n(\varepsilon_n,\eta_n)$ by $N(K_k,\bar{\Gamma}_n(\varepsilon_n,\eta_n))$. Since $w_{\max}\leq n^{1/(\tau-1)}/\eta_n$ with high probability, $\Gamma_n(\varepsilon_n,\eta_n)=M_n(\varepsilon_n)$ with high probability. Therefore, with high probability,
		\begin{equation}
			N\Big(K_k,\bar{M}_n\left(\varepsilon_n\right)\Big) = N\Big(K_k,\bar{\Gamma}_n(\varepsilon_n,\eta_n)\Big),
		\end{equation}
		where
                $N\big(K_k,\bar{M}_n\left(\varepsilon_n\right)\big)$
                denotes the number of maximal $k$-cliques on vertices
                not in $M_n\left(\varepsilon_n\right)$.
                By~\eqref{eq:expcontrsmall} and the Markov inequality,
                we have for all $\epsilon > 0$
		\begin{equation}
			\lim_{n \to \infty} \Prob{\left|\frac{N\Big(K_k,\bar{\Gamma}_n(\varepsilon_n,\eta_n)\Big)}{n^{(3-\tau)(2\tau-3)/(\tau-1)}}\right| > \epsilon} = 0.
		\end{equation}

		Furthermore, Lemma~\ref{lem:intfinite} combined with the lower bound in~\eqref{eq:expnhsmall} shows that when choosing $u_i=1/\varepsilon$ and $\ell_i=\varepsilon$ for  some fixed $\varepsilon>0$ for all $i$,
		\begin{align}
			\Exp{N(K_k,M_n(\varepsilon))}\in \Theta(n^{(3-\tau)(2\tau-3)/(\tau-1)}).
		\end{align}
		Thus, for fixed $\varepsilon>0$, 
		\begin{align} 
			N(K_k)&= N(K_k,M_n(\varepsilon))+N(K_k,\bar{M}_n(\varepsilon))=\Theta_p(n^{(3-\tau)(2\tau-3)/(\tau-1)}),
		\end{align}
  which
		shows that
		\begin{equation}
			\frac{N\big(K_k,M_n\left(\varepsilon_n\right)\big)}{N(K_k)}\plim 1,
		\end{equation}
		as required. This completes the proof of Theorem~\ref{thm:maxcliqeslocalized}.
	\end{proof}

	\section{Experiments}
	\label{sec:experiments}
	
	As mentioned in the introduction, empirical evidence suggests that
	the number of maximal cliques in IRGs and GIRGs is
	small~\cite{Exter_Valid_Avera_Analy_ESA2022}.  In fact, all
	generated networks with $n = \SI{50}{k}$ nodes and expected average
	degree $10$ have fewer maximal cliques than edges.  This stands in
	stark contrast to our super-polynomial lower bounds.  This
	discrepancy probably comes from the fact that $n = \SI{50}{k}$ is
	low enough that a linear lower-order term dominates the
	super-polynomial terms.  In this section, we complement our
	theoretical lower bounds with experiments\footnote{The corresponding
	code is available at:
	\url{https://github.com/thobl/maximal-cliques-scale-free-rand-graph}}
	with an $n$ that is sufficiently large to make the super-polynomial
	terms dominant.  Additionally, we consider dense and super-dense
	Erdős--Rényi graphs.
	
	\subsection{Cliques in the Dense Subgraph of GIRGs and IRGs}
	\label{sec:girgs-irgs}
	
	Our theoretical lower bounds are based on the existence of a dense
	subgraph among the vertices with weights $\Theta(\sqrt{n})$.  To
	experimentally observe the super-polynomial scaling, we generate IRGs
	and GIRGs restricted to vertices of high weight.  This restriction
	lets us consider much larger values of $n$.  In the following, we
	first describe the exact experiment setup, before describing and
	discussing the results.
	
	\paragraph{Experiment Setup.}
	
	We generate IRGs and GIRGs with varying number of vertices $n$ and
	deterministic power-law weights where the $v$th vertex has weight
	\begin{equation*}
		w_v = \left( \frac{n}{v} \right)^{\frac{1}{\tau - 1}}.
	\end{equation*}
	Note that the minimum weight is~$w_n = 1$.
	
	We use the power-law exponents $\tau \in \{2.2, 2.5, 2.8\}$ and for
	GIRGs we consider the temperatures $T \in \{0, 0.4, 0.8\}$ and
	dimension $d = 1$.  For each parameter setting, we consider two
	subgraphs: The subgraph induced by vertices with
	$0.5 \sqrt{n} \le w_i \le \sqrt{n}$ and within the larger interval
	$0.5 \sqrt{n} \le w_i \le n$.  In preliminary experiments, we also
	tried constant factors other than $0.5$, yielding comparable
	results.
	
	As connection probability for the IRGs between the $u$th and $v$th
	vertex, we use $\min\{1, w_u w_v / n\}$, i.e., vertices of
	weight $1$ have connection probability $1 / n$ and vertices of weight
	at least $\sqrt{n}$ are deterministically connected.  For GIRGs, we
	choose the constant factor $\mu$ in Equation~\eqref{eq:edgeprob} such that we obtain the same
	expected\footnote{We do not sample the positions before computing the
		expected average degree but we compute the expectation with respect
		to random positions.} average degree as for the corresponding IRG in
	the considered subgraph.
	For each of these configurations, we generate \num{10} graphs.
	Figure~\ref{fig:girg_irg_core_plot} shows the average.
	
	\begin{figure}[t]
		\centering
		\includegraphics{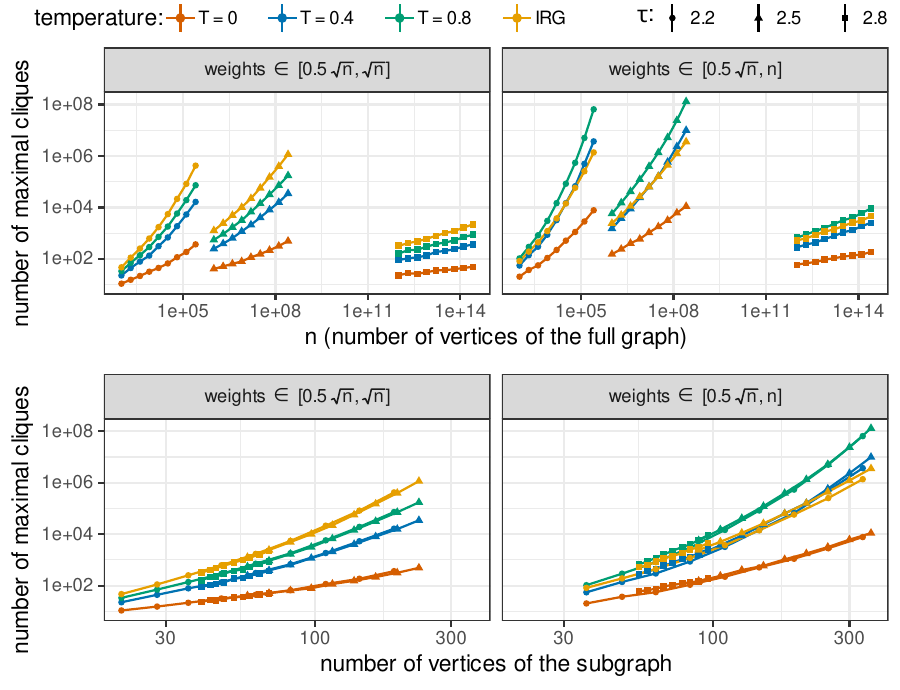}
		\caption{The number of maximal cliques of the dense subgraph of
			GIRGs and IRGs.  The considered subgraphs contain all vertices
			with weights in $[0.5 \sqrt{n}, \sqrt{n}]$ (left column) and
			$[0.5 \sqrt{n}, n]$ (right column).  The top and bottom plots show
			the number of cliques with respect to the size of the full graph,
			and with respect to the size of the considered subgraph,
			respectively.  All axes are logarithmic.  Each point is the
			average of \num{10} sampled graphs.}
		\label{fig:girg_irg_core_plot}
	\end{figure}
	
	\paragraph{General Observations.}
	
	One can clearly see in Figure~\ref{fig:girg_irg_core_plot} (top row)
	that the scaling of the number of cliques depending on the graph size
	is super-polynomial (upward curves in a plot with logarithmic axes).
	Thus, on the one hand, this agrees with our theoretical analysis.  On
	the other hand, the plots also explain why previous
	experiments~\cite{Exter_Valid_Avera_Analy_ESA2022} showed a small
	number of cliques: While the scaling is super-polynomial, the constant
	factors are quite low.  In the top-left plot for $\tau = 2.5$, more
	than \SI{200}{M} 
 nodes are necessary to get just barely above
	\SI{1}{M} maximal cliques in the dense subgraph.  For $\tau = 0.8$
	this is even more extreme with $n = \SI{200}{T}$ yielding only
	\SI{10}{k} maximal cliques.  Thus, unless we deal with huge graphs,
	the maximal cliques in the dense part of the graph are dominated by
	the number of cliques in the sparser parts, despite the
	super-polynomial growth of the former.
	
	\paragraph{Effect of the Power-Law Exponent $\tau$.}
	
	The top plots of Figure~\ref{fig:girg_irg_core_plot} show that a
	smaller power-law exponent $\tau$ leads to more maximal cliques.  The
	bottom plots show the number of cliques with respect to the size of
	the dense subgraph and not with respect to the size of the full graph.
	One can see that the difference for the different power-law exponents
	solely comes from the fact that the dense subgraph is larger for
	smaller $\tau$.  For the same size of the dense subgraph, the scaling
	is almost independent of the power-law exponent.
	
	\paragraph{Effect of the Geometry.}
	
	In the left plots of Figure~\ref{fig:girg_irg_core_plot}, we can see
	that geometry leads to fewer maximal cliques.  For $T = 0$, the
	super-polynomial scaling is only barely noticeable.  Higher
	temperatures lead to a larger number of cliques and we get even more
	cliques for IRGs.  Interestingly, the scaling is slower for IRGs when
	additionally considering the core of vertices with weight more than
	$\sqrt{n}$ (see next paragraph).
	
	\paragraph{Effect of the Core.}
	
	When not capping the weight at $\sqrt{n}$ but also considering
	vertices of even higher weight (right plots), we can observe the
	following.  The overall picture remains similar, with a slightly
	increased number of cliques.  However, this increase is higher for
	GIRGs than it is for IRGs.  A potential explanation for this is the
	following.  For IRGs, the core forms a clique and adding a large
	clique to the graph does not change the overall number of maximal
	cliques by too much.  For GIRGs, however, it depends on the constant
	$\mu$ controlling the average degree whether this subgraph forms a
	clique or not.  Thus, for the same average degree, the maximum clique
	is probably somewhat smaller for GIRGs and thus adding the vertices of
	weight at least $\sqrt{n}$ leads to more additional cliques than in
	IRGs.
	
	\subsection{Cliques in the Dense and Super-Dense Erdős--Rényi Graphs}
	\label{sec:cliq-dense-gnp}
	
	Here we count the cliques for dense Erdős--Rényi graphs with constant
	connection probabilities $p \in \{0.6, 0.7, 0.8, 0.9\}$ and
	super-dense Erdős--Rényi graphs with connection probability
	$p = 1 - c / n$ for $c \in \{1, 2, 4, 8\}$.  Note that the complement
	of a super-dense Erdős--Rényi graph has constant expected average
	degree.  The scaling of the number of cliques with respect to the
	number of vertices is shown in Figure~\ref{fig:dense_gnp_plot}, where
	each point represents \num{20} samples.
	
	\begin{figure}[t]
		\centering
		\includegraphics{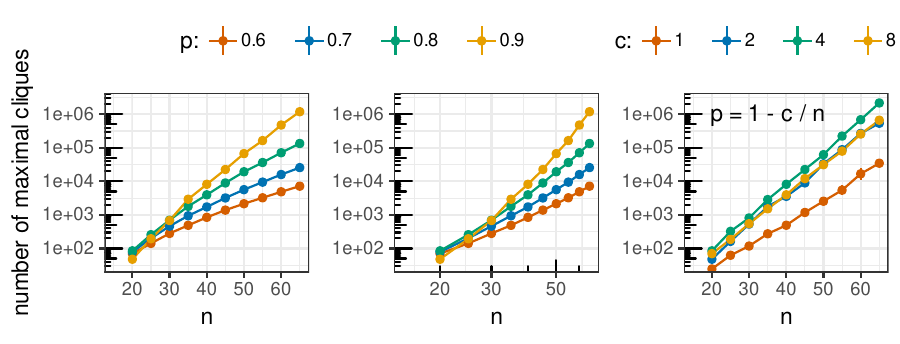}
		\caption{The number of maximal cliques in dense and super-dense
			$G(n, p)$s.  For the left and middle plot, $p$ is constant.  For
			the right plot, $p = 1 - c / n$ for constant $c$.  Note that the
			$y$-axes are logarithmic and the $x$-axis in the middle plot is
			logarithmic.  Each point is the average of \num{20} sampled
			graphs.}
		\label{fig:dense_gnp_plot}
	\end{figure}
	
	Note that for constant $p$, the left plot with logarithmic $y$-axis is
	curved downward, indicating sub-exponential scaling, while the middle
	plot with logarithmic $x$- and $y$-axis is bent upwards, indicating
	super-polynomial scaling.  This is in line with our lower bound in
	Theorem~\ref{thm:er-lower-bound}.
	
	For the super-dense case, the right plot indicates exponential
	scaling, in line with Theorem~\ref{thm:erdense}.

 \section{Conclusion and Discussion}\label{sec:conc}
In this paper, we have investigated the number of maximal cliques in three random graph models: the Erd\H{o}s--R\'enyi random graph, the inhomogeneous random graph and the geometric inhomogeneous random graph. We have shown that sparse  Erd\H{o}s--R\'enyi random graphs only contain a polynomial amount of maximal cliques, but in the other two sparse models, the number of maximal cliques scales at least super-polynomially in the network size. This is caused by the degree-heterogeneity in these models, as many large maximal cliques are present close to the core of these random graphs. We prove that there only exist a linear amount of small maximal cliques. Interestingly, these small maximal cliques are almost always formed by two low-degree vertices, whereas all other vertices are hubs of high degree. 

We have then shown that this dominant super-polynomial behavior of the number of maximal cliques often only kicks for extreme network sizes, and that experimentally, lower-order linear terms instead drive the scaling of the number of maximal cliques until large values of the network size. This explains the dichotomy between the theoretical super-polynomial lower bounds for these models, and the observation that in real-world networks, the amount of maximal cliques is often quite small. 

Several of our results only constitute lower bounds for the number of maximal cliques. We believe that relatively close upper bounds can be constructed in a similar fashion, but leave this open for further research. 

	While Theorem~\ref{thm:girg_non_torus} only holds for 2-norms, we believe that the theorem can be extended to any $L^p$-norm for $p\neq1,\infty$, by looking at the $L^p$ norm-cycle instead of the regular cycle. For $p=1,\infty$ this approach fails, shortest distance paths to non-opposing segments pass through the center of the cycle. Therefore, opposing segments are just as close as many non-opposing ones. Whether Theorem~\ref{thm:girg_non_torus} also holds with 1 or $\infty$ norms is therefore a question for further research.
	We also believe that this approach also extends to the underlying space $[0,1]^d$ for general $d$, where instead of looking at a cycle inside $[0,1]^2$, one studies a $d$-ball inscribed in $[0,1]^d$ instead.

 \bibliographystyle{abbrv}
	\bibliography{references}

	\appendix
	
	\section{Proof of Theorem~\ref{thmnonzero2dim}}\label{app:prooftemperature}

	\begin{lemma}\label{lem:chernoff-areas}
		Let $(A_i)_{i\in[k]}$ be a set of areas of size $A$, and let $S$ be a set of vertices, such that $A|S|>n^\varepsilon$ for some $\varepsilon>0$. Then, for any $0<\lambda<1$ and $k<\exp(\lambda A|S|)$, with high probability all areas contain at least $(1-\lambda)A|S|$ vertices.
	\end{lemma}
	\begin{proof}
		The Chernoff bound gives for the number of vertices from $S$ within area $A$, $N_{S,A}$:
		\begin{equation}
			\Prob{N_{S,A} < (1-\lambda)A|S| }\leq \exp\Big(-\lambda A|S|\Big).
		\end{equation}
		This implies that when $A|S|>n^{\varepsilon}$ for some $\varepsilon>0$, then, with high probability, all areas contain at least $(1-\lambda)A|S|$ vertices.
	\end{proof}

	We follow the same construction of areas and sets as in the proof of Theorem~\ref{thm:girg_non_torus}. By~\eqref{eq:mbound} this creates $2k = s \cdot n^{\min(\beta/2,\gamma(n)/2)} $  areas of size $A=n^{-3/2\gamma(n)}$, with on average $\Exp{|S|}=n^{(3-\tau)/2-\beta}$ vertices. Thus, Lemma~\ref{lem:chernoff-areas} shows that as long as $\beta+3/2\gamma(n) <(3-\tau)/2$, then all areas contain with high probability at least $$n'= c_1n^{(3-\tau)/2-\beta-3/2\gamma(n)}$$ vertices for some $c_1>0$.  
	
	From~\eqref{eq:edgeprob}, it follows that any set of vertices that contains one in each given area still satisfies the requirement that all vertices in non-opposite boxes connect, as in non-opposite boxes, the connection probability equals 1 by~\eqref{eq:l2dimnotorus}. Now to form a co-matching, vertices in opposite boxes should not connect. 
	
	With high probability, a positive proportion of vertices in two opposing areas have distance at least $t+h=a+n^{-\gamma(n)}$, by the uniform distribution within areas, and the fact that a positive proportion of the two areas have distance $t+h$.
	
	By~\eqref{eq:edgeprob}, the probability that vertices $i,j\in S$ at distance at least $a+n^{-\gamma(n)}$ are connected is bounded by
	\begin{align}\label{eq:opposite_disconnect_prob}
		p_{ij}& \leq \min\Bigg(\Big(\frac{a^2(1-n^{-\beta})}{(a+n^{-\gamma(n)})^2}\Big)^{1/T},1\Bigg)\nonumber\\
		& =(1-n^{-\beta})(1-n^{-\gamma(n)})(1+o(1))\nonumber\\
		& =(1-\max(n^{-\beta},n^{-\gamma(n)}))(1+o(1)).
	\end{align}
 Similarly as in~\eqref{eq:pm0},
 \begin{equation}
     \Prob{M=0}\leq k (1-\max(n^{-\beta},n^{-\gamma(n)})^{(n')^2}\leq k\exp(-\max(n^{-\beta},n^{-\gamma(n)})(n')^2)
 \end{equation}
 Using that $n'= c_1n^{(3-\tau)/2-\beta-3/2\gamma(n)}$ therefore yields
 \begin{equation}
     \Prob{M=0}\leq k\exp(-c_1^2n^{(3-\tau)-\beta-3\gamma(n)}\max(n^{-\beta},n^{-\gamma(n)})).
 \end{equation}
 Thus, choosing $\beta=\gamma(n)=(3-\tau)/5-\varepsilon$ ensures that there is a co-matching of size $k=s \cdot n^{(3-\tau)/10-\varepsilon}$
	\qed

\section{Proof of Lemma~\ref{lem:intfinite}}\label{app:proofintfinite}
 	\begin{proof}
		
		This integral equals
		\begin{align}
			&\int_0^1\dots\int_0^1x_3^{1-\tau}\cdots x_k^{1-\tau} \int_0^1\int_0^{x_2} x_1^{k-1-\tau}x_2^{k-1-\tau}\exp\Big(-\mu^{1-\tau}x_1x_2^{\tau-2}\Big)dx_1 dx_2\dots dx_k\nonumber\\
			&
			+ 
			\int_0^1\dots\int_0^1 x_3^{2-\tau}\cdots x_k^{2-\tau}\int_1^{\infty}\int_0^{x_2} x_1^{k-1-\tau}x_2^{1-\tau}\prod_{i=3}^k\min\Big(x_2x_i,1\Big)\exp\Big(-\mu^{1-\tau}x_1x_2^{\tau-2}\Big)dx_1 dx_2\dots dx_k
		\end{align}
		Now the first integral is bounded by
		\begin{equation}
			\int_0^1\dots\int_0^1 x_3^{2-\tau}\cdots x_k^{2-\tau}\int_0^1\int_0^{x_2} x_1^{k-1-\tau}x_2^{k-1-\tau}dx_2dx_1\dots dx_k<\infty,
		\end{equation}
		as $2-\tau>-1$, and $k-1-\tau>-1$ for $k\geq 3$ as well. We now turn to the second integral. 
		The second integral is finite if
		\begin{align}
			&\int_0^1\dots\int_0^1 x_3^{1-\tau}\cdots x_k^{1-\tau}\int_1^{\infty} \int_0^{x_2}x_1^{k-1-\tau}x_2^{1-\tau}\prod_{i=3}^k\min\Big(x_2x_i,1\Big)\ind{x_2^{\tau-2}x_1<1}dx_1 dx_2\dots dx_k<\infty.
		\end{align}
		
		This results in 
		\begin{align}\label{eq:int2ind}
			& \int_0^1\dots\int_0^1 x_3^{1-\tau}\cdots x_k^{1-\tau} \int_1^{\infty} \int_0^{x_2^{2-\tau}}x_1^{k-1-\tau}x_2^{1-\tau}\prod_{i=3}^k\min\Big(x_2x_i,1\Big) dx_1 \dots dx_k\nonumber\\
			& = \int_0^1\dots\int_0^1 x_3^{1-\tau}\cdots x_k^{1-\tau}\int_1^{\infty}x_2^{(2-\tau)(k+1-\tau)-1}\prod_{i=3}^k\min\Big(x_2x_i,1\Big) dx_2\dots dx_k
		\end{align}
		W.l.o.g. we assume that $x_3>x_4>\dots>x_k$. Then, the inner integral evaluates to
		\begin{align}
			&  \int_1^{\infty}x_2^{(2-\tau)(k+1-\tau)-1}\prod_{i=3}^k\min\Big(x_2x_i,1\Big) dx_2\nonumber\\
			& = \int_1^{1/x_3}x_2^{(2-\tau)(k+1-\tau)+k-3}x_3\cdots x_k dx_2+ \dots +  \int_{1/x_k}^\infty x_2^{(2-\tau)(k+1-\tau)-1} dx_2\nonumber\\
			& = C_3x_3^{(\tau-2)(k+1-\tau)+3-k}x_4\cdots x_k + C_4x_4^{(\tau-2)(k+1-\tau)+4-k}x_5\cdots x_k+ \dots + C_{k} x_k^{(\tau-2)(k+1-\tau)}
		\end{align}
		We now show that all these terms evaluate to a finite integral when plugged into~\eqref{eq:int2ind}. Indeed,
		
		\begin{align}
			& \int_0^1\int_0^{x_3}\dots\int_0^{x_{k-1}}x_l^{(\tau-2)(k+1-\tau)+l-k}x_{l+1}\cdots x_k x_3^{1-\tau}\cdots x_k^{1-\tau} dx_k dx_{k-1}\dots dx_3\nonumber\\
			& = \int_0^1\int_0^{x_3}\dots\int_0^{x_{l-1}}x_l^{(\tau-2)(l-\tau)-1} x_3^{1-\tau}\cdots x_{l-1}^{1-\tau} dx_l dx_{l-1}\dots dx_3\nonumber\\
			& = \int_0^1\int_0^{x_3}\dots\int_0^{x_{l-2}}x_{l-1}^{(\tau-2)(l-1-\tau)-1} x_3^{1-\tau}\cdots x_{l-2}^{1-\tau} dx_l dx_{l-2}\dots dx_3<\infty
		\end{align}
		as the index $l-k$ remains at least 3. Therefore,~\eqref{eq:intmaxclique} is finite as well. 
	\end{proof}

\end{document}